\documentclass[12pt, hyperref]{article}
\usepackage{bbding}
\usepackage{pifont}
\usepackage{}
\usepackage{mathrsfs}
\usepackage{amssymb}
\usepackage{amsfonts}
\usepackage{amssymb,latexsym,amsmath, amsthm}
\usepackage{hyperref}
\usepackage{amscd}
\usepackage[all]{xy}
\usepackage{xcolor}
\hypersetup{linkcolor=blue, pdfstartview=FitH} \hypersetup{colorlinks, linkcolor=blue,
pdfstartview=FitH}
\newtheorem{corollary}{Corollary}
\newtheorem{lemma}[corollary]{Lemma}
\newtheorem{definition}[corollary]{Definition}
\newtheorem{proposition}[corollary]{Proposition}
\newtheorem{theorem}[corollary]{Theorem}
\newtheorem{example}[corollary]{Example}
\newtheorem{remark}[corollary]{Remark}

\begin{document}

\title{$L^2$ estimates and existence theorems for the $\overline{\partial}$ operators in infinite dimensions, I}
\author{Jiayang Yu\thanks{School of Mathematics, Sichuan University, Chengdu 610064, P. R. China. {\small\it E-mail:} {\small\tt
 jiayangyu@scu.edu.cn}.} \,\quad\,and\,\quad\,Xu Zhang\thanks{School of Mathematics, Sichuan University, Chengdu 610064, P. R. China. {\small\it E-mail:} {\small\tt
zhang$\_$xu@scu.edu.cn}.}}
\date{}
\maketitle
\def\cc{\mathbb{C}}
\def\zz{\mathbb{Z}}
\def\nn{\mathbb{N}}
\def\rr{\mathbb{R}}
\def\qq{\mathbb{Q}}
\def\dd{\mathbb{D}}
\def\tt{\mathbb{T}}
\def\bb{\mathbb{B}}
\def\ff{\mathbb{F}}

\def\A{\mathcal{A}}
\def\B{\mathcal{B}}
\def\D{\mathcal{D}}
\def\L{\mathcal{L}}
\def\M{\mathcal{M}}
\def\Mperp{\mathcal{M}^\perp}
\def\N{\mathcal{N}}
\def\K{\mathcal{K}}
\def\F{\mathcal{F}}
\def\R{\mathcal{R}}
\def\s{\mathcal{S}}
\def\p{\mathcal{P}}
\def\P{\mathcal{P}}
\def\T{\mathcal{T}}
\def\O{\mathcal{O}}
\def\Z{\mathcal{Z}}

\def\cod{\text{cod}}
\def\fd{\text{fd}}
\def\ker{\text{ker}}
\def\ran{\text{ran}}
\def\mp{\text{mp}}

\def\al{\alpha}
\def\la{\lambda}
\def\ep{\epsilon}
\def\sig{\sigma}
\def\Sig{\Sigma}
\def\cd{\mathbb{C}^d}
\def\bm{\mathcal{M}}
\def\bn{\mathcal{N}}
\def\hN{H^2\otimes \mathbb{C}^N}
\def\ba{\mathcal{A}}
\def\hm{H^2\otimes \mathbb{C}^m}
\def\mb{\mathcal{M}^{\perp}}
\def\pr{{\mathbb C}[z_1,\cdots,z_n]}
\def\nb{\mathcal{N}^{\perp}}
\def\hrd{H^2(\mathbb{D}^n)}
\def\be{\mathcal{E}}
\def\po{{\mathbb C}[z,\,w]}
\def\hr{H^2({\mathbb D}^2)}
\def\bigpa#1{\biggl( #1 \biggr)}
\def\bigbracket#1{\biggl[ #1 \biggr]}
\def\bigbrace#1{\biggl\lbrace #1 \biggr\rbrace}

\def\papa#1#2{\frac{\partial #1}{\partial #2}}
\def\dbar{\bar{\partial}}

\def\oneover#1{\frac{1}{#1}}

\def\meihua{\bigskip \noindent $\clubsuit \ $}
\def\blue#1{\textcolor[rgb]{0.00,0.00,1.00}{#1}}
\def\red#1{\textcolor[rgb]{1.00,0.00,0.00}{#1}}

\def\norm#1{||#1||}
\def\inner#1#2{\langle #1, \ #2 \rangle}

\def\divide{\bigskip \hrule \bigskip}

\def\bigno{\bigskip \noindent}
\def\medno{\medskip \noindent}
\def\smallno{\smallskip \noindent}
\def\bignobf#1{\bigskip \noindent \textbf{#1}}
\def\mednobf#1{\medskip \noindent \textbf{#1}}
\def\smallnobf#1{\smallskip \noindent \textbf{#1}}
\def\nobf#1{\noindent \textbf{#1}}
\def\nobfblue#1{\noindent \textbf{\textcolor[rgb]{0.00,0.00,1.00}{#1}}}

\def\vector#1#2{\begin{pmatrix}  #1  \\  #2 \end{pmatrix}}

\def\cfh{\mathfrak{CF}(H)}

\begin{abstract}
The classical $L^2$ estimate for the $\overline{\partial}$ operators  is a basic tool in complex analysis of several variables. Naturally, it is expected to extend this estimate to infinite dimensional complex analysis, but this is a longstanding unsolved problem, due to the essential difficulty that there exists no nontrivial translation invariance measure in the setting of infinite dimensions. The main purpose in this series of work is to give an affirmative solution to the above problem, and apply the estimates to the solvability of the infinite dimensional $\overline{\partial}$ equations. In this first part, we focus on the simplest case, i.e., $L^2$ estimates and existence theorems for the $\overline{\partial}$ equations on the whole space of $\ell^p$ for $p\in [1,\infty)$. The key of our approach is to introduce a suitable working space, i.e., a Hilbert space for $(s,t)$-forms on $\ell^p$ (for each nonnegative integers $s$ and $t$), and via which we define the $\overline{\partial}$ operator from $(s,t)$-forms to $(s,t+1)$-forms and establish the exactness of these operators, and therefore in this case we solve a problem which has been open for nearly forty years.
\end{abstract}

\tableofcontents


\section{Introduction}

One of the most fundamental problems in finite and infinite dimensional complex analysis is the solvability of the inhomogeneous Cauchy-Riemann equation, or the following $\overline{\partial}$ equation in various (complex) spaces:
\begin{eqnarray}\label{d-bar equation}
\overline{\partial}u=f.
\end{eqnarray}
Typically, the equation \eqref{d-bar equation} is an over-determined system of partial differential equations with constant coefficients except for the very special case of one (complex) space dimension.

There are extensive works addressing to the equation \eqref{d-bar equation} in finite dimensions. Early studies on or related to this topic can be found in Weyl \cite{Weyl}, Kodaira \cite{Kodaira53}, Garabedian and Spencer \cite{Gar-Spencer}, Kohn and Spencer \cite{Koh-Spencer}, Morrey \cite{Mor}, Kohn \cite{Koh, Koh2}, H\"ormander \cite{Hor65} and so on. Further results can be found in Berndtsson \cite{Berndtsson}, Folland and Kohn \cite{Folland-Kohn}, Greiner and Stein \cite{Greiner-Stein},  H\"ormander \cite{Hor90}, Lieb and Michel \cite{Lieb-Michel}, Ohsawa \cite{Ohsawa}, Straube \cite{Straube} and rich refernces cited therein.

A basic approach to solve the $\overline{\partial}$ equation \eqref{d-bar equation} in finite dimensions is to derive {\it a priori} $L^2$ estimates, i.e., suitable integral  type estimates in some $L^2$ spaces. It originated from the classical Fredholm alternative and Carleman estimate. A rudiment of such an approach is available in Kodaira \cite{Kodaira53}. Then, this approach was systematically developed by Morrey \cite{Mor}, Kohn \cite{Koh},
H\"omander \cite{Hor65}, and others. A standard reference on this topic is H\"ormander \cite{Hor90} (See also two recent very interesting books Ohsawa \cite{Ohsawa} and Straube \cite{Straube}). It is notable that, $L^2$ estimates for the $\overline{\partial}$ operator may be applied to solve many problems in complex analysis of several variables, complex geometry and algebraic geometry (e.g., Diederich and Ohsawa \cite{Diederich-Ohsawa}, Donnelly and Fefferman \cite{Donnelly-Fefferman}, Demailly \cite{Demailly06, Demailly}, Guan and Zhou \cite{Guan-Zhou}, Siu \cite{Siu1974, Siu1998}, just to mention a few).

Starting from the last sixties, infinite dimensional complex analysis is a rapid developing field, in which one can find many works, for example, the early survey by Nachbin \cite{Nachbin}, the monographes by Colombeau \cite{Col}, Dineen \cite{Din99}, Herv\'e \cite{Herve}, Mazet \cite{Maz}, Mujica \cite{Mujica} and Noverraz \cite{Noverraz}, and the rich references cited therein.  Around 2000, Lempert revisited this field and made significant progresses on both Dolbeault complex and plurisubharmonic domination in infinite dimensions in a series of important works \cite{Lem98, Lem99, Lem00, Lem03}.

Also, there exist some studies on the solvability of the $\overline{\partial}$ equation \eqref{d-bar equation} in infinite dimensions. Three representative works are as follows:
\begin{itemize}
\item As far as we know, Henrich \cite{Hen} (1973) is the first paper on this topic, in which the existence of solutions of $(\ref{d-bar equation})$ under some polynomial
growth conditions was proved. More precisely, Henrich \cite[Theorem 1.4.1]{Hen} showed that, if $B$ and $H$ are respectively Banach space and Hilbert space so that $H\subset B$ and $(H,B)$ is an abstract Wiener space (introduced by Gross \cite{Gro65}), and $f$ is a closed $(0,t)$-form ($t\in \mathbb{N}$ ) on $B$ with polynomial growth and bounded Fr\'{e}chet derivative on bounded subsets of $B$, then it is possible to find a $(0,t-1)$-form $u$ solving the $\overline{\partial}$ equation \eqref{d-bar equation} on $H$.

\item In 1979, Raboin \cite{Rab79} proved the solvability of $(\ref{d-bar equation})$ in a separable Hilbert space $H$, in which $f$ is a smooth $(0,1)$-form; while its solution $u$ is restricted on a very special subspace rather than any open subset of $H$ so that the (translation) quasi-invariance of the underlying Gaussian measure can be employed there (See also Colombeau \cite{Col} and Soraggi \cite{Sor} for some further results by Raboin's method). Note that, all solutions found in the works \cite{Hen, Rab79, Col, Sor} are only defined in some smaller subspaces than that for the inhomogeneous term $f$.

\item About twenty years later, in 1999, a breakthrough was made by Lempert, who showed in \cite[Theorem 1.1]{Lem99} that, if $f$ is a $\overline{\partial}$ closed $(0,1)$-form which is Lipschitz continuous on any proper sub-open balls of an open ball $B(R)$ (centered at $0$) with radius $R\leq \infty$, then $(\ref{d-bar equation})$ admits a continuously differentiable solution $u$ on $B(R)$ (See \cite{Lem04} for further results in this respect). To the best of our knowledge, this is the first work in infinite dimensional Banach spaces, for which the inhomogeneous term $f$ and the solution $u$ (in the $\overline{\partial}$ equation \eqref{d-bar equation}) are defined in the same space.
\end{itemize}

Now, although more than twenty years have passed after Lempert's work \cite{Lem99}, very little more has been known about the solvability of $(\ref{d-bar equation})$. The situation is actually very similar as Lempert remarked at the very beginning of \cite{Lem99} that, ``{\sl up to now not a single Banach space and an open subset therein have been proposed where the equation $(\ref{d-bar equation})$ could be proved to be solvable under reasonably general conditions on $f$}". Here, we  mention the following two aspects:
\begin{itemize}
\item
In general, it seems that the spaces of continuous functions are not the best working spaces for solving the equation $(\ref{d-bar equation})$. Indeed, Coeur\'{e} gave a counterexample (appeared in \cite{Maz}) of $(0,1)$-form $f$ of the class $C^1$ on $\ell^2$ for which the equation $(\ref{d-bar equation})$ has no (continuous) solution on any nonempty open sets. Then, Lempert in \cite[Theorem 9.1]{Lem99} extended Coeur\'{e}'s counterexample to that on $\ell^p$ for any $p\in \mathbb{N}$, i.e., a $(0,1)$-form $f$ of the class $C^{p-1}$ on $\ell^p$ for which $(\ref{d-bar equation})$ has no solution on any nonempty open sets. Further, Payti \cite{Pat} constructs a Banach space $X$ and a $(0,1)$-form of class $C^{\infty}$ on $X$ such that $(\ref{d-bar equation})$ does not admit any local solution around $0$.

\item No solvability result is published for the equation $(\ref{d-bar equation})$ in infinite dimensions with $f$ being a general $(s,t)$-form (for any given nonnegative integers $s$ and $t$). In 1982, Colombeau said at \cite[p. 430]{Col} that, ``{\it the case of $(s,t)$-forms with $t>1$ remains unsolved}". Now, nearly forty years have passed but the situation does not change too much. Indeed, many works after Colombeau concerned at most the case of $(0,t)$-forms, but there is no work addressing the general case of $(s,t)$-forms with $s>0$.
\end{itemize}

The main purpose in this series of work is to develop $L^2$ estimates to solve the $\overline{\partial}$ equation \eqref{d-bar equation} in various infinite dimensional spaces. Roughly speaking, we shall assume only the closedness and the $L^2$ integrality of the inhomogeneous term $f$. It deserve mentioning that we do not need any continuity condition on $f$.   Also, $f$ may be a general $(s,t)$-form. Moreover, all of the solutions $u$ in \eqref{d-bar equation}  are defined in the same space as that for $f$.

Because of the great success of the classical $L^2$ estimates in finite dimensions, it is quite nature to expect that such sort of estimates can be extended to the infinite dimensional setting, but this is actually a longstanding unsolved problem, due to the essential difficulty that there is no nontrivial translation invariance measure in infinite dimensions. Indeed, it seems that Raboin's original purpose in \cite{Rab79} was exactly to try to do this but finally he failed, and therefore, as mentioned before, the solution $u$ (of \eqref{d-bar equation}) that he found is defined in some space which is smaller than that for $f$. After Raboin's work, though more than forty years have passed, the
situation does not change before this work.

Another essential difficulty in infinite dimensions is how to treat the  differential (especially the higher order differential) of a (possibly vector-valued) function. In this respect, the most popular notions are  Fr\'echet's derivative and G\^ateaux's derivative, which coincide whenever one of them is continuous. However, higher order (including the second order) Fr\'echet's derivatives are multilinear mappings, which are usually not easy to handle analytically. For example, it is well-known that, for any given
infinite dimensional Hilbert space $H$, although ${\cal L}(H)$, the space of all linear bounded operators on $H$ (which is isomorphic to the space of all bounded bilinear functionals on $H$), is still a Banach space, it is neither
reflexive (needless to say to be a Hilbert space) nor separable anymore even if $H$ itself is
separable. Because of this, it may be quite difficult to handle some analytic problems related to ${\cal L}(H)$ (e.g., L\"u and Zhang \cite{Lu-Zhang}).

 In this paper, we shall focus on the simplest case, i.e., $L^2$ estimates and existence theorems for the $\overline{\partial}$ equation \eqref{d-bar equation} on the whole space of $\ell^p$, for $p\in [1,\infty)$. The main contributions in this paper are as follows:
 \begin{itemize}
\item
 For each nonnegative integers $s$ and $t$, we  introduce a suitable Hilbert space, i.e., $L^2_{(s,t)}(\ell^p)$ for $(s,t)$-forms on $\ell^p$ (See Subsection \ref{subsec2.5}), which is our working space. Denote by ${\cal L}_{(s,t)}(\ell^p)$ the space of all bounded $s+t$-multilinear functional-valued, continuous and bounded functions on $\ell^p$, in which $s$ is {\it the order of $\mathbb{C}$-linearity} and $t$ is {\it the order of conjugate $\mathbb{C}$-linearity}. Clearly, $L^2_{(s,t)}(\ell^p)$ can be regarded as a weak version of ${\cal L}_{(s,t)}(\ell^p)$ (This point is quite similarly to the case of the usual $L^2(\mathbb{R})$ and the space $C_b(\mathbb{R})$ of all continuous and bounded functions on $\mathbb{R}$). Nevertheless, the main difference between $L^2_{(s,t)}(\ell^p)$ and ${\cal L}_{(s,t)}(\ell^p)$ is that the former is a Hilbert space while the later is only a Banach space (again, neither
reflexive nor separable anymore, whenever $s+t>0$). Clearly, our definitions for $(s,t)$-forms are natural generalization of that for differential forms in finite dimensions; they are quite convenient for computation and studying some analysis problems in infinite dimensional spaces, and hence, we believe that they have some independent interest and may be applied in other places.

\item Then, similarly to the setting in finite dimensions, we define the $\overline{\partial}$ operator from  $(s,t)$-forms on $\ell^p$ into $(s,t+1)$-forms on $\ell^p$ (See \eqref{d-dar-f-defi}), in which we use only the usual partial derivatives in the weak sense, and therefore, Fr\'echet's derivative or G\^ateaux's derivative is completely avoided in our definition.

\item Further, we obtain a key dimension-free $L^2$ estimates for solutions of a family of $\overline{\partial}$ equations in finite dimensions (See Corollary \ref{200208t2}), and via which we obtain a full characterization on the solvability  of the $\overline{\partial}$ equation \eqref{d-bar equation} on $\ell^p$, in the $L^2$-level (See Corollary \ref{uniqueness theorem}). Indeed, the counterexamples constructed by Coeur\'{e} on $\ell^2$ and by Lempert on $\ell^p$ (for $p\in\mathbb{N}$), that we mentioned before, are solvable for the $\overline{\partial}$ equation in the sense of this paper (See Example \ref{xmple1} for more details). On the other hand, since our results hold for any $(s,t)$-forms on $\ell^p$, in this case we solve the aforementioned, longstanding open problem posed at \cite[p. 430]{Col}.
\end{itemize}

Extension of our approach to more general setting, say the $\overline{\partial}$ operators in pseudo-convex domains in infinite dimensions or in infinite dimensional manifolds, as well as more applications, will be given in our forthcoming works.

To end this introduction, let us mention a very interesting correspondence between complex analysis and (mathematical) control theory (e.g., Li and Yong \cite{Li-Yong}, Zhang \cite{Zhangxu} and Zuazua \cite{Zu1}), i.e., the $\overline{\partial}$ equations correspond to the controllability problems, and the $L^2$ estimates to the observability estimates (See Remark \ref{remark2-9-21} for more details).

\section{Preliminaries}
\subsection{Some facts from functional analysis}
In this subsection, we shall collect some facts on unbounded operators on Hilbert spaces, which will play some key roles on establishing the existence of  solutions to the $\overline{\partial}$ equations.

Let $H_1$ and $H_2$ be two complex Hilbert spaces and let $T$ be a densely defined, linear operator from $H_1$ into $H_2$. Denote by $D_T$ the domain of $T$, and write $$R_T\triangleq\{Tx:x\in D_T\},\quad N_T\triangleq\{x\in D_T:Tx=0\},
$$
which are the range and the kernel of $T$, respectively.


In the rest of this subsection, we assume that $T$ is a linear, closed and densely defined operator from $H_1$ into $H_2$. Then, $T^*$ is also a linear, closed and densely defined operator from $H_2$ into $H_1$, and $T^{**}=T$ (e.g., \cite[Section 4.2]{Kra}).

By \cite[Lemma 4.1.1]{Hor90}  and the proof of \cite[Lemma 3.3]{SlU}, it is easy to show the following result:
\begin{lemma}\label{lower bounded lemma}
Let $F$ be a closed subspace of $H_2$ and $R_T\subset F$. Then $F=R_T$ if and only if for some constant $C>0$,
\begin{eqnarray}\label{lower bounded}
||g||_{H_2}\leq C||T^*g||_{H_1},\quad \forall\; g\in D_{T^*}\cap F.
\end{eqnarray}
In this case, for any $f\in F$, there is a unique $u\in D_T\cap N_T^\bot$ such that
\begin{eqnarray}\label{abs equa}
Tu=f,
\end{eqnarray}
and
 \begin{eqnarray}\label{lower bounded1}
 ||u||_{H_1}\le C||f||_{H_2},
 \end{eqnarray}
where the constant $C$ is the same as that in \eqref{lower bounded}.
\end{lemma}

\begin{proof}
The first assertion (for the characterization of $F=R_T$) is given by \cite[Lemma 4.1.1]{Hor90}.  Similarly to the proof of \cite[Lemma 3.3]{SlU}, one can prove the second assertion except the uniqueness. If there is another $u'\in D_T\cap N_T^\bot$ such that $Tu'=f$, then $T(u-u')=Tu-Tu'=0$, and therefore $u-u'\in N_T\cap N_T^\bot=\{0\}$, which yields $u=u'$, and hence, the desired uniqueness follows. This completes the proof of Lemma \ref{lower bounded lemma}.
\end{proof}

\begin{remark}\label{estimation rem}
Of course, generally speaking, the solutions of \eqref{abs equa} are NOT unique. Indeed, the set of solutions of \eqref{abs equa} is as follows:
$$
U\equiv\{u+v:\,v\in N_T\},
$$
where $u$ is the unique solution (in the space $D_T\cap N_T^\bot$) to this equation (found in Lemma \ref{lower bounded lemma}). It is obvious that $u\bot N_T$.
\end{remark}

Similarly to \cite[Lemma 3.3]{SlU}, as a consequence of Lemma \ref{lower bounded lemma}, we have the following result.

\begin{corollary}\label{estimation lemma}
Let $S$ be a linear, closed and densely defined operator from $H_2$ to another Hilbert space $H_3$, and $R_T\subset N_S$. Then,

{\rm 1)} $R_T=N_S$ if there exists a constant $C>0$ such that
\begin{eqnarray}\label{estimation lower bounded}
||g||_{H_2}\leq C\sqrt{||T^*g||_{H_1}^2+||Sg||_{H_3}^2},\quad \forall\; g\in D_{T^*}\cap D_S;
\end{eqnarray}

{\rm 2)} $R_T=N_S$ if and only if for some constant $C>0$,
\begin{eqnarray}\label{weakestimation lower bounded}
||g||_{H_2}\leq C||T^*g||_{H_1},\quad \forall\; g\in D_{T^*}\cap N_S.
\end{eqnarray}

In each of the above two cases, for any $f\in N_S$, there exists a unique $u\in D_T\cap N_T^\bot$  such that $Tu=f$ and
\begin{eqnarray}\label{Solution estation}
||u||_{H_1}\leq C||f||_{H_2},
 \end{eqnarray}
where the constant $C$ is the same as that in \eqref{estimation lower bounded} (or \eqref{weakestimation lower bounded}).
\end{corollary}

\begin{remark}\label{remark2-9-1}
It seems that the estimate \eqref{estimation lower bounded} is equivalent to \eqref{weakestimation lower bounded}, i.e., \eqref{estimation lower bounded} is also a necessary condition for $R_T=N_S$ but, as far as we know, this point is still unclear.
\end{remark}

\begin{remark}\label{remark2-9-21}
Obviously, the point of Lemma \ref{lower bounded lemma} and Corollary \ref{estimation lemma} is to reduce the solvability of the underlying equation problem to a suitable {\it a priori} estimate for its dual problem. Such sort of reductions are frequently used in other mathematical branches. For example, in control theory, people usually reduce the controllability problem (which can be viewed as an equation problem in which both the state and the control variables are unknowns) to a suitable observability estimate for its dual problem. Clearly, Lemma \ref{lower bounded lemma} is very close to Li and Yong \cite[Lemma 2.4 of Chaper 7, p. 282]{Li-Yong}. On the other hand, many observability estimates for concrete problems (especially that for deterministic/stochastic partial differential equations) can be viewed as some sort of $L^2$ estimates; moreover, a basic tool to establish observability estimates is also to use Carleman type estimates (e.g., Zhang \cite{Zhangxu} and Zuazua \cite{Zu1}).
\end{remark}

\subsection{A Borel probability measure on $\ell^p$}\label{subsect2.2}
In this subsection, for any $p\in [1,\infty)$, we will introduce a Borel probability measure on $\ell^p$ (to be defined later). It is a sort of Gaussian measures on Banach spaces, but here we adopt the method of product measures (e.g., \cite[Chapter 1]{Da}) instead of that of abstract Wiener spaces (See \cite{Gro65}), so that the involved computation in the sequel are made more friendly.

If $(G,\cal J)$ is a topological space (on a given
nonempty set $G$), then the smallest $\sigma$-algebra
generated by all open sets in $\cal J$ is called the
Borel $\sigma$-algebra of
$G$, denoted by $\mathscr{B}(G)$. For any given $a>0$, we define a probability measure $\bn_a$ in $(\mathbb{R}^2,\mathscr{B}(\mathbb{R}^2))$ as follows:
\begin{eqnarray*}
\bn_a(B)\triangleq \frac{1}{2\pi a^2}\int_{B}e^{-\frac{x^2+y^2}{2a^2}}\,\mathrm{d}x\mathrm{d}y,\quad\,\forall\;B\in \mathscr{B}(\mathbb{R}^2).
\end{eqnarray*}
Since $\mathbb{C}$ can be identified with $\mathbb{R}^2$, $\bn_a$ is also a probability measure in $(\mathbb{C} ,\mathscr{B}(\mathbb{C} ))$.

In the rest of this paper, we will fix a sequence $\{a_i\}_{i=1}^{\infty}$ of positive numbers  such that
\begin{eqnarray*}
\sum_{i=1}^{\infty} a_i <1.
\end{eqnarray*}
For each $n\in\mathbb{N}$, we define a probability measure $\bn^n$ in $(\mathbb{C}^n,\mathscr{B}(\mathbb{C}^n))$ by setting
\begin{eqnarray}\label{def of N^n}
\bn^n\triangleq \prod_{i=1}^{n}\bn_{a_i}.
\end{eqnarray}
Now we want to define a product measure on the space $\mathbb{C}^{\infty}=\prod\limits_{i=1}^{\infty}\mathbb{C}$ which can be identified with $(\mathbb{R}^2)^{\infty}=\prod\limits_{i=1}^{\infty}\mathbb{R}^2$. We endow $\mathbb{C}^{\infty}$ with the usual product topology.  By the discussion in \cite[Section 1.5]{Da}, $\mathscr{B}(\mathbb{C}^{\infty})$ is precisely the product $\sigma$-algebra  generated by the following family of sets (in $\mathbb{C}^{\infty}$):
\begin{eqnarray*}
A_1\times A_2\times \cdots \times A_k\times \Bigg(\prod_{l=k+1}^{\infty}\mathbb{C}\Bigg),
\end{eqnarray*}
where $A_i\in \mathscr{B}(\mathbb{C}),\,1\leq i\leq k, \, k\in \mathbb{N}$.
Let
\begin{eqnarray*}
\bn\triangleq\prod_{i=1}^{\infty}\bn_{a_i}
\end{eqnarray*}
be the product measure on $(\mathbb{C}^{\infty},\mathscr{B}(\mathbb{C}^{\infty}))$. This is the very measure that we will use later.

For $p\in[1,\infty)$, let
\begin{eqnarray*}
\ell^p\triangleq \bigg\{(x_n+\sqrt{-1}y_n)\in \mathbb{C}^{\infty}:x_n,y_n\in\mathbb{R},\,n\in\mathbb{N},\,\sum_{n=1}^{\infty}(|x_n|^p+|y_n|^p)<\infty\bigg\}.
\end{eqnarray*}
Also, let
\begin{eqnarray*}
\ell^{\infty}\triangleq \bigg\{(x_n+\sqrt{-1}y_n)\in \mathbb{C}^{\infty}:x_n,y_n\in\mathbb{R},\,n\in\mathbb{N},\,\sup_{1\leq n<\infty}\max\{|x_n|,|y_n|\}<\infty\bigg\}.
\end{eqnarray*}
Then $\ell^p$ ($p\in[1,\infty)$) is a complex separable Banach space. Just as \cite[Exercise 1.10]{Da}, one can show that every closed ball in $\ell^p$ lies in $\mathscr{B}(\mathbb{C}^{\infty})$ and
$\mathscr{B}(\ell^p)=\{A\cap\ell^p:A\in \mathscr{B}(\mathbb{C}^{\infty})\}$ for $p\in[1,\infty)$ and $\ell^{\infty}\in \mathscr{B}(\mathbb{C}^{\infty}).$

We shall need the following simple but useful result:

\begin{proposition}\label{lem1}
$\bn(\ell^p)=1$ for all $p\in[1,\infty]$.
\end{proposition}

\begin{proof}
For any $p\in[1,\infty)$, similarly to the proof of \cite[Proposition 1.11]{Da}, we have
\begin{eqnarray*}
\int_{\mathbb{C}^{\infty}}\sum_{i=1}^{\infty}(|x_i|^p+|y_i|^p)\,\mathrm{d}\bn
&=&\sum_{i=1}^{\infty}\int_{\mathbb{C}^{\infty}}(|x_i|^p+|y_i|^p)\,\mathrm{d}\bn\\
&=&\frac{2^{\frac{p}{2}+1}\Gamma(\frac{1+p}{2})}{\sqrt{\pi}}\sum_{i=1}^{\infty}a_i^p\\
&\leq&\frac{2^{\frac{p}{2}+1}\Gamma(\frac{1+p}{2})}{\sqrt{\pi}}\sum_{i=1}^{\infty}a_i<\infty.
\end{eqnarray*}
Here and henceforth, $\Gamma(\cdot)$ stands for the usual $\Gamma$-function.

Also,
\begin{eqnarray*}
\int_{\mathbb{C}^{\infty}}\sup_{1\leq n<\infty}\max\{|x_n|,|y_n|\}\,\mathrm{d}\bn\leq \int_{\mathbb{C}^{\infty}}\sum_{n=1}^{\infty}(|x_n|+|y_n|)\,\mathrm{d}\bn<\infty.
\end{eqnarray*}
Hence, $\bn(\ell^p)=1$ for all $p\in[1,\infty]$. This competes the proof of Proposition \ref{lem1}.
\end{proof}

\begin{remark}
Proposition \ref{lem1} with $p=2$ was given in \cite[Proposition 1.11]{Da}. The case with $p=1$ was communicated privately to us by Haimeng Luo.
\end{remark}

Thanks to Proposition \ref{lem1}, for any fixed $p\in[1,\infty)$, we obtain a probability measure $P$ on $(\ell^p,\mathscr{B}(\ell^p))$ by setting
\begin{eqnarray*}
P(A)\triangleq \bn(A),\qquad\forall\,A\in \mathscr{B}(\ell^p).
\end{eqnarray*}
We denote by $L^2(\ell^p,P)$ the Hilbert space of all equivalence classes of square integrable complex-valued functions on $\ell^p$, endowed with the following inner product,
\begin{eqnarray*}
\langle f,g\rangle_{L^2(\ell^p,P)} =\int_{\ell^p}f\cdot\overline{g}\,\mathrm{d}P,\qquad\forall\,f,g\in L^2(\ell^p,P).
\end{eqnarray*}

The following result can be viewed as a special version of Fernique's theorem for abstract Wiener spaces (\cite{Fer}).
\begin{lemma}\label{Ferniqeu's theorem}
Let $p\in[1,\infty)$ and $\varphi\in (\ell^p)^*$. Then, for
any $\epsilon \le\frac{1}{||\varphi||_{(\ell^p)^*}}$ if $\varphi\not=0$ and $\epsilon<\infty $ if $\varphi=0$,
it holds that
\begin{eqnarray*}
\int_{\ell^p}e^{\epsilon |\varphi(\textbf{z})|}\,\mathrm{d}P(\textbf{z})\leq\int_{\ell^1}e^{ ||\textbf{z}||_{\ell^1}}\,\mathrm{d}P(\textbf{z})<\infty.
\end{eqnarray*}
\end{lemma}
\begin{proof}
By $\varphi\in (\ell^p)^*$, for any $p\in[1,\infty)$,
$$
|\varphi(\textbf{z})|\le ||\varphi||_{(\ell^p)^*}||\textbf{z}||_{\ell^p}\le  ||\varphi||_{(\ell^p)^*}||\textbf{z}||_{\ell^1},\quad\forall\;\textbf{z}\in\mathbb{C}^{\infty}.
$$
Hence, by Proposition \ref{lem1}, for any $\epsilon \le\frac{1}{||\varphi||_{(\ell^p)^*}}$ if $\varphi\not=0$ and $\epsilon<\infty $ if $\varphi=0$,
\begin{eqnarray*}
&&\int_{\ell^p}e^{\epsilon |\varphi(\textbf{z})|}\,\mathrm{d}P(\textbf{z})=\int_{\ell^1}e^{\epsilon |\varphi(\textbf{z})|}\,\mathrm{d}P(\textbf{z})\\
&&\leq \int_{\ell^1}e^{ ||\textbf{z}||_{\ell^1}}\,\mathrm{d}P(\textbf{z})=\int_{\mathbb{C}^{\infty}}e^{ ||\textbf{z}||_{\ell^1}}\,\mathrm{d}\mathcal{N}(\textbf{z})
=\lim_{n\to\infty}\int_{\mathbb{C}^{\infty}}e^{ \sum\limits_{k=1}^{n}(|x_k|+|y_k|)}\,\mathrm{d}\mathcal{N}(\textbf{z}),
\end{eqnarray*}
where the last equality follows from the monotone convergence theorem. Note that
\begin{eqnarray*}
&&\int_{\mathbb{C}^{\infty}}e^{ \sum\limits_{k=1}^{n}(|x_k|+|y_k|)}\,\mathrm{d}\mathcal {N}(\textbf{z})
=\int_{\mathbb{C}^{n}}e^{ \sum\limits_{k=1}^{n}(|x_k|+|y_k|)}\,\mathrm{d}\mathcal {N}^n
=\prod_{k=1}^{n}\int_{\mathbb{C}}e^{ |x_k|+|y_k|)}\,\mathrm{d}\mathcal {N}_{a_k}(x_k,y_k)\\
&&=\prod_{k=1}^{n}\bigg(\frac{2}{\sqrt{2\pi a_k^2}}\int_0^{+\infty}e^{-\frac{(x-a_k^2)^2}{2a_k^2}+\frac{a_k^2}{2}}\,\mathrm{d}x\bigg)^2
= \prod_{k=1}^{n}\bigg[ e^{\frac{a_k^2}{2}}\cdot\bigg(1+\frac{2}{\sqrt{\pi}}\int_0^{\frac{a_k}{\sqrt{2}}}e^{-x^2}\,\mathrm{d}x\bigg)\bigg]^2\\
&&\leq\prod_{k=1}^{n}\bigg[ e^{\frac{a_k^2}{2}}\cdot\bigg(1+\frac{\sqrt{2}}{\sqrt{\pi}}\cdot a_k\bigg)\bigg]^2
=e^{\sum\limits_{k=1}^{n}a_k^2}\cdot \prod_{k=1}^{n}\bigg(1+\frac{\sqrt{2}}{\sqrt{\pi}}\cdot a_k\bigg)^2
\leq e^{\sum\limits_{k=1}^{n}\Big(a_k^2+\frac{2\sqrt{2}}{\sqrt{\pi}}\cdot a_k\Big)},
\end{eqnarray*}
where the last inequality follows the simple fact that for any positive integer $l$ and complex numbers $u_1,\cdots,u_l$, it holds that $\prod\limits_{k=1}^{l}(1+|u_k|)\leq e^{\sum\limits_{k=1}^{l}|u_k|}$ (e.g., \cite[Lemma 15.3]{Rud87}). Therefore,
\begin{eqnarray*}
\int_{\ell^p}e^{\epsilon |\varphi(\textbf{z})|}\,\mathrm{d}P(\textbf{z})&=&\lim_{n\to\infty}\int_{\mathbb{C}^{\infty}}e^{ \sum\limits_{k=1}^{n}(|x_k|+|y_k|)}\,\mathrm{d}\mathcal {N}(\textbf{z})\\
&\leq &\lim_{n\to\infty} e^{\sum\limits_{k=1}^{n}\Big(a_k^2+\frac{2\sqrt{2}}{\sqrt{\pi}}\cdot a_k\Big)}=e^{\sum\limits_{k=1}^{\infty}\Big(a_k^2+\frac{2\sqrt{2}}{\sqrt{\pi}}\cdot a_k\Big)}<\infty.
\end{eqnarray*}
This completes the proof of Lemma \ref{Ferniqeu's theorem}.
\end{proof}

\subsection{Some dense subsets of $L^2(\ell^p,P)$}
In order to define ``weak derivatives", we need some family of ``smooth" functions on $\ell^p$ ($p\in [1,\infty)$), which are dense in $L^2(\ell^p,P)$. We need to borrow some idea from the unpublished manuscript \cite{Dri}.

Throughout this subsection we will assume that $X$ is a real separable Banach space\footnote{Clearly, any complex Banach space is automatically a real Banach space.}, and denote by $X^*$ the (real) dual space of $X$. 

Denote by $C_c^{\infty}(\mathbb{R}^n)$ the set of all $C^{\infty}$ real-valued functions on $\mathbb{R}^n$ with compact support. We need to use the following two sets of functions later.
$$
 \begin{array}{ll}
\mathcal {F}C_c^{\infty}(X)\triangleq \{f=F(\varphi_1,\cdots,\varphi_n):\; n\in\mathbb{N}, F\in C_c^{\infty}(\mathbb{R}^n), \varphi_i\in X^*, 1\leq i\leq n\}.
\end{array}
$$
We denote by $\mathscr {F}C_c^{\infty}(X)$ the complex linear space spanned by $\mathcal {F}C_c^{\infty}(X)$.

We need the following elementary result:
\begin{lemma}\label{rectangle lemma}
Let $R=[a_1,b_1]\times\cdots\times [a_n,b_n]$ for some $a_i,b_i\in\mathbb{R}$ satisfying $a_i<b_i$ ($i=1,2,\cdots,n$).  Then there exists uniformly bounded $\{f_k\}_{k=1}^{\infty}\subset C_c^{\infty}(\mathbb{R}^n)$ such that $f_k\to \chi_{R}$ pointwise in $\mathbb{R}^n$ as $k\to\infty$, where $\chi_{R}$ stands for the characteristic function of the set $R$.
\end{lemma}
\begin{proof}
It suffices to consider the one dimensional case, i.e., $R=[a,b]$ for some $a,b\in\mathbb{R}$ satisfying $a<b$. Clearly, one can find a sequence of continuous and uniformly bounded  functions $\{\varsigma_k\}_{k=1}^\infty$ with compact supports such that $\varsigma_k \to \chi_{R}$ pointwise in $\mathbb{R}$ as $k\to\infty$. Let $\varphi\in C_c^{\infty}(\mathbb{R})$ be such that $\varphi\geq 0$, $\hbox{supp}\,\varphi\subset [-1,1]$ and $\int_{\mathbb{R}}\varphi(x)\,\mathrm{d}x=1.$ Set $\varphi_{\epsilon}\triangleq\frac{1}{\epsilon}\varphi(\frac{x}{\epsilon}),\,\epsilon>0,x\in\mathbb{R}.$ Write $\varsigma_k^\epsilon\equiv\varphi_{\epsilon}\ast \varsigma_k$, where $\ast$ is the usual convolution operation. Then, $\varsigma_k^\epsilon\in C_c^{\infty}(\mathbb{R})$ and, for each $k\in \mathbb{N}$, $\varsigma_k^\epsilon\to \varsigma_k$ uniformly in $\mathbb{R}$ as $\epsilon\to 0$. Now, the desired result follows from the standard diagonal argument. This completes the proof of Lemma \ref{rectangle lemma}.
\end{proof}
We also need the following known approximation result (e.g., \cite[Lemma 3.4.6]{Coh}).
\begin{lemma}\label{approximation lemma}
Let $(X,\mathscr{A},\mu)$ be a finite measure space, and let $\mathscr{A}_0$ be an  algebra of subsets of $X$ such that $\mathscr{A}$ is precisely the $\sigma$-algebra generated by $\mathscr{A}_0$. Then for for each $A\in\mathscr{A}$ and positive number $\epsilon$ there is a set $A_0\in\mathscr{A}_0$ satisfying that $\mu(A\Delta A_0)<\epsilon$.
\end{lemma}
We denote by $\hbox{\rm span}\;\{e^{\sqrt{-1}\varphi}:\varphi\in (\ell^p)^*\}$ the complex linear space spanned by the set $\{e^{\sqrt{-1}\varphi}:\varphi\in (\ell^p)^*\}$ (In this subsection, $(\ell^p)^*$ stands for the real dual space of $\ell^p$). We have the following result:
\begin{proposition}\label{dense prop}
Both $\mathscr {F}C_c^{\infty}(\ell^p)$ and $\hbox{\rm span}\;\{e^{\sqrt{-1}\varphi}:\varphi\in (\ell^p)^*\}$ are dense in $L^2(\ell^p,P)$.
\end{proposition}
\begin{proof}
Denote by $\overline{\mathscr {F}C_c^{\infty}(\ell^p)}$ the closure of $\mathscr {F}C_c^{\infty}(\ell^p)$ in $L^2(\ell^p,P)$, and by $\mathscr{A}_0$ the algebra  generated by the following family of subsets (in $\ell^p$):
$$
\begin{array}{ll}
\big\{(x_j+\sqrt{-1}y_j)\in \ell^p:&(x_i,y_i)\in [a_i,b_i]\times [c_i,d_i],\, i=1,\cdots,k\big\},
\end{array}
$$
where $a_i,b_i,c_i,d_i\in \mathbb{R}$ satisfying $-\infty <a_i<b_i<+\infty, -\infty <c_i<d_i<+\infty$ ($i=1,\cdots,k$), $k\in \mathbb{N}$.
Also, for each $i\in \mathbb{N}$ , define two functions $\phi_i$ and $\psi_i$ on $\ell^p$ as follows:
$$
\phi_i((x_j+\sqrt{-1}y_j))=x_i,\quad
\psi_i((x_j+\sqrt{-1}y_j))=y_i,\quad
\forall\;(x_j+\sqrt{-1}y_j)\in \ell^p.
$$
It is clear that $\phi_i,\,\psi_i\in (\ell^p)^*$. Note that $\overline{\mathscr {F}C_c^{\infty}(\ell^p)}$ is a complex linear space which is closed under multiplication. Hence, combining the fact that $\mathscr{B}(\ell^p)=\{A\cap\ell^p:A\in \mathscr{B}(\mathbb{C}^{\infty})\}$,  and by Lemmas \ref{rectangle lemma} and \ref{approximation lemma}, we conclude that $\chi_B\in\overline{\mathscr {F}C_c^{\infty}(\ell^p)}$ for each $B\in\mathscr{B}(\ell^p)$. Therefore, $\overline{\mathscr {F}C_c^{\infty}(\ell^p)}=L^2(\ell^p,P)$.

Denote by $\overline{\text{span}\;\{e^{\sqrt{-1}\varphi}:\varphi\in (\ell^p)^*\}}$  the closure of span$\;\{e^{\sqrt{-1}\varphi}:\varphi\in (\ell^p)^*\}$ in $L^2(\ell^p,P)$. Suppose that $f\in C_c^{\infty}(\mathbb{R})$ and for $M\gg 1$, let $f_M(x)\triangleq \sum_{k\in\mathbb{Z}}f(x+2\pi kM),\,\,x\in \mathbb{R}.$ Then $f_M$ is a $2\pi M$-period continuous function on $\mathbb{R}$ and therefore by the Weierstrass approximation theorem (e.g., \cite[Theorem 5.7]{Rud91}), there exist complex polynomials $p_m(\xi,\overline{\xi})$ for $\xi\in\mathbb{C}$ and $m\in \mathbb{N}$ such that $p_m(e^{\sqrt{-1}\frac{x}{M}},e^{-\sqrt{-1}\frac{x}{M}})$ converges to $f_M(x)$ uniformly in $x$ as $m\to\infty$. In particular this implies that for any $\varphi\in (\ell^p)^*$, $f_M(\varphi)$ is the uniform limit of $p_m(e^{\sqrt{-1}\frac{\varphi}{M}},e^{-\sqrt{-1}\frac{\varphi}{M}})$  as $M\to\infty$. Since $f_M(\varphi)\to f(\varphi)$ as $M\to\infty$, it follows that $f(\varphi)\in\overline{\text{span}\;\{e^{\sqrt{-1}\varphi}:\varphi\in (\ell^p)^*\}}$. Similarly, for any $n\in\mathbb{N}$, $F\in C_c^{\infty}(\mathbb{R}^n)$ and $\varphi_i\in (\ell^p)^*$ ($1\leq i\leq n$), one can show that $F(\varphi_1,\cdots,\varphi_n)\in\overline{\text{span}\;\{e^{\sqrt{-1}\varphi}:\varphi\in (\ell^p)^*\}}$. Hence, $\overline{\mathscr {F}C_c^{\infty}(\ell^p)}\subset\overline{\text{span}\;\{e^{\sqrt{-1}\varphi}:\varphi\in (\ell^p)^*\}}$. This completes the proof of Proposition \ref{dense prop}.
\end{proof}

We denote by $\mathbb{N}^{(\mathbb{N})}$ the set of all finitely supported sequences of nonnegative integers. For $m\in\mathbb{N}$, set $\mathbb{N}^m\triangleq\{\alpha\in \mathbb{N}^{(\mathbb{N})}:\alpha=(\alpha_1,\cdots,\alpha_m,0,\cdots,0,\cdots),\,\alpha_i\in\mathbb{N},\,1\leq i\leq m\}$. Then $\mathbb{N}^{(\mathbb{N})}=\bigcup\limits_{m=1}^{\infty}\mathbb{N}^m.$ For $\textbf{z}=(z_k)=(x_k+iy_k)\in \mathbb{C}^{\infty}$, $\alpha=(\alpha_i),\,\beta=(\beta_i)\in \mathbb{N}^{m}$, write $$\textbf{z}^{\alpha} \overline{\textbf{z}}^{\beta}\triangleq z_1^{\alpha_1}\overline{z_1}^{\beta_1}z_2^{\alpha_2}\overline{z_2}^{\beta_2}\cdots z_m^{\alpha_m}\overline{z_m}^{\beta_m},
 $$
which is called a monomial on $\mathbb{C}^{\infty}$. Obviously, $\textbf{z}^{\alpha} \overline{\textbf{z}}^{\beta}$ can be viewed as a function on $\ell^p$ or $\mathbb{C}^m$. Set $$\mathcal {P}\triangleq\text{span}\;\{ \textbf{z}^{\alpha}{\overline{\textbf{z}}}^{\beta}:\alpha,\beta\in \mathbb{N}^{(\mathbb{N})}\},\quad\mathcal {P}_m\triangleq\text{span}\;\{ \textbf{z}^{\alpha}{\overline{\textbf{z}}}^{\beta}:\alpha,\beta\in \mathbb{N}^{m}\}.
$$
One has the following result.

\begin{proposition}\label{Desity Theorem II}
For any $p\in[1,\infty)$, $\mathcal{P}=\bigcup\limits_{m=1}^{\infty}\mathcal{P}_m$, $\mathcal{P}$ is dense in $L^2(\ell^p,P)$ and $\mathcal {P}_n$ is dense in $L^2(\mathbb{C}^n,\mathcal {N}^n)$ for each $n\in\mathbb{N}$. $L^2(\ell^p,P)$ and $L^2(\mathbb{C}^n,\mathcal {N}^n)$ are separable Hilbert spaces.
\end{proposition}
\begin{proof}
Firstly, we prove that the following set:
\begin{equation}\label{deff}
\begin{array}{ll}
\mathcal {F}\triangleq\{Q(\varphi_1,\cdots,\varphi_n):&n\in\mathbb{N},\varphi_s\in (\ell^p)^*,1\leq s\leq n, Q\hbox{ is a polynomial with}\\[3mm]
&n\hbox{ variables and complex coefficients}\}.
\end{array}
\end{equation}
is dense in $L^2(\ell^p,P)$. We use the contradiction argument. If $\mathcal {F}$ was not dense in $L^2(\ell^p,P)$, then, by the Riesz representation theorem, there would exist a non-zero $g\in L^2(\ell^p,P)$ such that
\begin{eqnarray} \label{zero condition}
\int_{\ell^p}fg\,\mathrm{d}P=0,\quad\forall\;f\in\mathcal {F}.
\end{eqnarray}
We will arrive at contradiction by showing that $g=0.$  Let $\varphi\in (\ell^p)^*$, by Lemma \ref{Ferniqeu's theorem}, there exists $\epsilon>0$, depending only on the norm of $\varphi$,  such that $\int_{\ell^p}e^{\epsilon |\varphi|}\,\mathrm{d}P<\infty.$ Then for any $n\in\mathbb{N}_0\equiv \mathbb{N}\cup\{0\}$ and $\delta\in(0,\frac{\epsilon}{3}),$
$$\int_{\ell^p}|\varphi|^ne^{\delta |\varphi|}|g|\,\mathrm{d}P \leq ||g||_{L^2(\ell^2,P)}\cdot \Big(\int_{\ell^p}|\varphi|^{2n}e^{2\delta |\varphi|} \,\mathrm{d}P  \Big)^{\frac{1}{2}}<\infty,$$
since there exists a constant $C>0$, depending only on $\epsilon>0$,  such that $|\varphi|^{2n}e^{2\delta |\varphi|}\leq Ce^{\epsilon |\varphi|}$. Define
$$
F(z)\triangleq \int_{\ell^p}e^{z\varphi}g\,\mathrm{d}P,\quad \forall\; z\in\mathbb{C}\hbox{ with }|\hbox{\rm Re}\;z|<\delta.
$$
Then $F(z)$ is analytic on $\{z\in\mathbb{C}:\,|\text{Re}z|<\delta\}$ and $F^{(n)}(z)=\int_{\ell^p}\varphi^ne^{z\varphi}g\,\mathrm{d}P$ for all $n\in\mathbb{N}_0.$ In particular, $F^{(n)}(0)=\int_{\ell^p}\varphi^n g\,\mathrm{d}P$ for all $n\in\mathbb{N}_0$ and hence $F^{(n)}(0)=0$ by (\ref{zero condition}). Since $\{z\in\mathbb{C}:\,|\text{Re}z|<\delta\}$ is a connected open set, by the uniqueness theorem for analytic functions of one variable (See \cite[Theorem 10.18]{Rud87}) it then follows that $F(z)\equiv 0$ on $\{z\in\mathbb{C}:\,|\text{Re}z|<\delta\}$ and in particular $F(\sqrt{-1})=\int_{\ell^p}e^{\sqrt{-1}\varphi}g\,\mathrm{d}P=0$. Since $\varphi\in (\ell^p)^*$ is arbitrary we have shown that $\int_{\ell^p}e^{\sqrt{-1}\varphi}g\,\mathrm{d}P=0$ for all $\varphi\in (\ell^p)^*.$ By Proposition \ref{dense prop}, span$\;\{e^{\sqrt{-1}\varphi}:\varphi\in (\ell^p)^*\}$ is dense in $L^2(\ell^p,P)$ this implies that $g=0$ and we have arrive at the desired contradiction and hence the set $\mathcal {F}$ defined by \eqref{deff} is dense in $L^2(\ell^p,P)$.

For $\varphi_1,\cdots,\varphi_n\in (\ell^p)^*$, each $1\leq s\leq n$ and $m\in\mathbb{N}$, set
$$
\varphi_s^m(\textbf{z})=\varphi_s(z_1,\cdots,z_m,0,0,\cdots),\quad\forall\,\textbf{z}=(z_k)\in \ell^p,
$$
and $Q_m\triangleq Q(\varphi_1^m,\cdots,\varphi_n^m)$. Then by $p\in[1,\infty)$, using Lemma \ref{Ferniqeu's theorem} and the dominated control convergence theorem we have $Q_m\to Q(\varphi_1,\cdots,\varphi_n)$ in $L^2(\ell^p,P)$ and note that $Q_m\in\mathcal{P}$ for each $m$ which proves that  $\mathcal {P}$ is dense in $L^2(\ell^p,P)$. Similarly discussion may deduce that $\mathcal {P}_n$ is dense in $L^2(\mathbb{C}^n,\mathcal {N}^n)$. The last claim of  Proposition \ref{Desity Theorem II} follows from that the members with rational coefficients of $\mathcal {P}$ and $\mathcal {P}_n$ are countable dense subsets in $L^2(\ell^p,P)$ and $L^2(\mathbb{C}^n,\mathcal {N}^n)$ correspondingly.  The proof of Proposition \ref{Desity Theorem II} is completed.
\end{proof}

\begin{remark}\label{identify remark}
Note that  each $f\in\mathcal{P}_n$  can viewed as a function on $\ell^p$ or $\mathbb{C}^n$ and it holds that $\int_{\ell^p}|f|^2\,\mathrm{d}P=\int_{\mathbb{C}^n}|f|^2\,\mathrm{d}\mathcal {N}^n$. By Proposition \ref{Desity Theorem II}, $L^2(\mathbb{C}^n,\mathcal {N}^n)$ can be identified as a closed subspace of $L^2(\ell^p,P)$. Thus in the sequel, we will regard $L^2(\mathbb{C}^n,\mathcal {N}^n)$ as a closed subspace of $L^2(\ell^p,P)$ when there is no confusion.
\end{remark}

For each $n\in\mathbb{N}$, we denote by $C_c^{\infty}((\mathbb{R}^2)^n)$ all $C^{\infty}$ real-valued functions on $(\mathbb{R}^2)^n$ with compact support. Since $\mathbb{C}^n$ can be identified with $(\mathbb{R}^2)^n$, each $f\in C_c^{\infty}((\mathbb{R}^2)^n)$ can be viewed as a function on $\mathbb{C}^n$. Then $f$ can also be viewed as a cylinder function on $\mathbb{C}^{\infty}$ or $\ell^p$. Set
$$\mathscr {C}_c^{\infty}(\mathbb{C}^n)\triangleq \{f+\sqrt{-1}g:f,g\in C_c^{\infty}((\mathbb{R}^2)^n)\},\quad\mathscr {C}_c^{\infty}\triangleq \bigcup_{n=1}^{\infty}\mathscr {C}_c^{\infty}(\mathbb{C}^n).
$$
By Remark \ref{identify remark}, we can view $\mathscr {C}_c^{\infty}(\mathbb{C}^n)$ as a subspace of $L^2(\mathbb{C}^n,\mathcal {N}^n)$ and $\mathscr {C}_c^{\infty}$ is a subset of $L^2(\ell^p, P)$. A classical mollification argument shows that $\mathscr {C}_c^{\infty}(\mathbb{C}^n)$ is dense in $L^2(\mathbb{C}^n,\mathcal {N}^n)$. Hence, by Proposition \ref{Desity Theorem II}, we obtain the following result.
\begin{proposition}\label{smooth dense functions}
For any $p\in[1,\infty)$, $\mathscr {C}_c^{\infty}$ is a dense subset of $L^2(\ell^p, P)$.
\end{proposition}

\subsection{Definition of the $\overline{\partial}$ operator on $L^2(\ell^p,P)$}
Suppose that $f$ is a complex-valued function defined on $\ell^p$. For each $\textbf{z}=(z_j)=(x_j+\sqrt{-1}y_j),\,\textbf{z}^1=(z_j^1)=(x_j^1+\sqrt{-1}y_j^1)\in \ell^p$, where $x_j,y_j,x_j^1,y_j^1\in \mathbb{R}$ for each $j$, we define
\begin{eqnarray*}
&&\displaystyle\frac{\partial f(\textbf{z})}{\partial x_j}\triangleq\lim_{\mathbb{R}\ni t\to 0}\frac{f(z_1,\cdots,z_{j-1},z_j+t,z_{j+1},\cdots)-f(\textbf{z})}{t},\\[3mm]
&&\displaystyle\frac{\partial f(\textbf{z})}{\partial y_j}\triangleq\lim_{\mathbb{R}\ni t\to 0}\frac{f(z_1,\cdots,z_{j-1},z_j+\sqrt{-1}t,z_{j+1},\cdots)-f(\textbf{z})}{t},
\end{eqnarray*}
provided that the limits exist in some sense.
Just like the case of finite dimensions, we set $\overline{\textbf{z}}=(\overline{z_j})=(x_j-\sqrt{-1}y_j)$, and
\begin{eqnarray*}
 \frac{\partial f(\textbf{z})}{\partial z_j}\triangleq\frac{1}{2}\bigg(\frac{\partial f(\textbf{z})}{\partial x_j}-\sqrt{-1}\frac{\partial f(\textbf{z})}{\partial y_j}\bigg),\qquad
 \frac{\partial f(\textbf{z})}{\partial \overline{z_j}}\triangleq\frac{1}{2}\bigg(\frac{\partial f(\textbf{z})}{\partial x_j}+\sqrt{-1}\frac{\partial f(\textbf{z})}{\partial y_j}\bigg).
\end{eqnarray*}
Then, inspired by Lempert's definition \cite{ Lem99}, we {\it formally} define a function $df(\cdot,\cdot)$ on $\ell^p\times \ell^p$ by
\begin{eqnarray*}
df(\textbf{z},\textbf{z}^1)=\sum_{j=1}^{\infty} \bigg(\frac{\partial f(\textbf{z})}{\partial x_j}x_j^1+\frac{\partial f(\textbf{z})}{\partial y_j}y_j^1\bigg)=\sum_{j=1}^{\infty} \bigg(\frac{\partial f(\textbf{z})}{\partial z_j}z_j^1+\frac{\partial f(\textbf{z})}{\partial \overline{z_j}}\overline{z_j^1}\bigg).
\end{eqnarray*}

Now, for any $f\in\mathscr {C}_c^{\infty}$, we write
\begin{eqnarray}\label{200209dd1}
\overline{\partial} f(\textbf{z},\textbf{z}^1)\triangleq \sum_{j=1}^{\infty}\frac{\partial f(\textbf{z})}{\partial \overline{z_j}}\overline{z_j^1},
\end{eqnarray}
which is a well-defined function from $\ell^p\times \ell^p$ into $\mathbb{C}$, and it is conjugate $\mathbb{C}$-linear on the second variable $\textbf{z}^1$. By Proposition \ref{smooth dense functions}, $\mathscr {C}_c^{\infty}$ is a dense subset of $L^2(\ell^p,P)$ and obviously $\overline{\partial}f\in L^2(\ell^p\times \ell^p,P\times P)$ for each $f\in\mathscr {C}_c^{\infty}$. In this way, we defined an operator $\overline{\partial}$ from the dense subset $\mathscr {C}_c^{\infty}$ of $L^2(\ell^p,P)$ into  $L^2(\ell^p\times \ell^p,P\times P)$.

Now we extend the definition of $\overline{\partial}$ weakly. For simplicity of notation, for any smooth function $\varphi=\varphi(\textbf{z})$ on $\mathbb{C}^{\infty}$ and $j\in\mathbb{N}$, we write
\begin{equation}\label{delta-j}\delta_j\varphi\equiv\frac{\partial \varphi}{\partial z_j}-\frac{\overline{z_j}}{2a_j^2}\varphi,\quad\forall\;\textbf{z}=(z_j).
\end{equation}
\begin{definition}\label{200209dd2}
For $f\in L^2(\ell^p, P)$, if there exists a sequence $\{f_j\}_{j=1}^{\infty}\subset L^2(\ell^p, P)$ such that
\begin{eqnarray*}
\int_{\ell^p}f_j\cdot\overline{\varphi}\,\mathrm{d}P
=-\int_{\ell^p}f \cdot\overline{\delta_j\varphi}\,\mathrm{d}P,\quad\forall\;\varphi\in \mathscr{C}_c^{\infty}
\end{eqnarray*}
and
$$\sum\limits_{j=1}^{\infty}2a_j^2  \int_{\ell^p} |f_j|^2\,\mathrm{d}P  <\infty,
$$
then, we define $\overline{\partial}f(\cdot,\cdot)\in L^2(\ell^p\times \ell^p,P\times P)$ as follows:
\begin{eqnarray*}
\overline{\partial}f(\textbf{z},\textbf{z}^1)\triangleq\sum\limits_{j=1}^{\infty}f_j(\textbf{z})\overline{z_j^1} .
\end{eqnarray*}
\end{definition}
\begin{remark}
Clearly, for any $f\in \mathscr {C}_c^{\infty}$, $\overline{\partial} f(\textbf{z},\textbf{z}^1)$ given by Definition \ref{200209dd2} coincides that by \eqref{200209dd1}. Since $\mathscr {C}_c^{\infty}$ is dense in $L^2(\ell^p, P)$, we see that $\overline{\partial}$ is a densely defined linear operator.
\end{remark}
\begin{lemma}\label{db is closable}
The operator $\overline{\partial}$ given by Definition \ref{200209dd2} is a densely defined, closed linear operator from $ L^2(\ell^p,P)$ to $L^2(\ell^p\times \ell^p,P\times P)$.
\end{lemma}
\begin{proof}
Choose any sequence $\{f_n\}_{n=1}^{\infty}\subset D_{\overline{\partial}}$ satisfying that $f_n\to f$ in $L^2(\ell^p,P)$ and $\overline{\partial}f_n\to F$ in $L^2(\ell^p\times \ell^p,P\times P)$ for some $F\in L^2(\ell^p\times \ell^p,P\times P)$. It suffices to show that $f\in D_{\overline{\partial}}$ and $\overline{\partial}f=F$.  Note that
\begin{eqnarray*}
\mathscr {L}\equiv\bigg\{G\triangleq \sum\limits_{j=1}^{\infty}g_j(\textbf{z})\overline{z_j^1}:\;\{g_j\}_{j=1}^{\infty}\subset L^2(\ell^p,P)\;\hbox{ and }\;\sum_{j=1}^{\infty}2a_j^2 \int_{\ell^p} |g_j|^2\,\mathrm{d}P<\infty\bigg\}
\end{eqnarray*}
is a closed subspace of $L^2(\ell^p\times \ell^p,P\times P)$, and for any $G\triangleq \sum\limits_{j=1}^{\infty}g_j(\textbf{z})\overline{z_j^1}\in \mathscr {L}$,
 $$
 ||G||_{L^2(\ell^p\times \ell^p,P\times P)}=\sqrt{\sum_{j=1}^{\infty}2a_j^2 \int_{\ell^p} |g_j|^2\,\mathrm{d}P}.
 $$
Clearly, $\{f_n\}_{n=1}^{\infty}\subset D_{\overline{\partial}}$ implies that $\overline{\partial}f_n\in \mathscr {L}$ for each $n\in \mathbb{N}$ and there exists $\{f_n^j\}_{n,j=1}^{\infty}\subset L^2(\ell^p,P)$ such that $\overline{\partial}f_n=\sum\limits_{j=1}^{\infty}f_n^j(\textbf{z})\overline{z_j^1}$, and
\begin{eqnarray}\label{Integration by part}
\int_{\ell^p} f_n^j\cdot\overline{\varphi}\,\mathrm{d}P
=-\int_{\ell^p}f_n \cdot\overline{\delta_j\varphi} \,\mathrm{d}P,\quad\forall\;\varphi\in \mathscr{C}_c^{\infty}.
\end{eqnarray}

By $\overline{\partial}f_n\to F$ in $L^2(\ell^p\times \ell^p,P\times P)$, it follows that $F\in \mathscr {L}$. Hence, there exists $\{F_j\}_{j=1}^{\infty}\subset L^2(\ell^p,P)$ such that $F=\sum\limits_{j=1}^{\infty}F_j(\textbf{z})\overline{z_j^1}$ and
\begin{eqnarray*}
\int_{\ell^p\times\ell^p}|\overline{\partial} f_n-F|^2\,\mathrm{d}(P\times P)=\sum_{j=1}^{\infty}2a_j^2\int_{\ell^p} |  f_n^j-F_j |^2\,\mathrm{d}P,\quad\forall\,n\in\mathbb{N},\,
\end{eqnarray*}
which implies that for each $j\in\mathbb{N}$, $f_n^j\to F_j$ in $L^2(\ell^p,P)$ as $n\to\infty$.
Letting $n\to\infty$ in \eqref{Integration by part}, we obtain that
\begin{eqnarray*}
\int_{\ell^p} F_j\cdot\overline{\varphi}\,\mathrm{d}P
=-\int_{\ell^p}f \cdot\overline{\delta_j\varphi} \,\mathrm{d}P,\quad\forall\;\varphi\in \mathscr{C}_c^{\infty}.
\end{eqnarray*}
Therefore, $f\in D_{\overline{\partial}}$ and $\overline{\partial}f=F$. This completes the proof of Lemma \ref{db is closable}.
\end{proof}

\subsection{An extension of the $\overline{\partial}$ operator to $(s,t)$-forms on $\ell^p$}\label{subsec2.5}
Suppose that $s,t$ are non-negative integers. If $s+t\ge1$, write
$$
(\ell^p)^{s+t}=\underbrace{\ell^p\times\ell^p\times\cdots\times\ell^p}_{s+t\hbox{ \tiny times}}.
$$
Suppose that $I=(i_1,\cdots,i_s)$ and $J=(j_1,\cdots,j_t)$ are multi-indices, where $i_1,\cdots,i_s,$ $j_1,\cdots,j_t\in \mathbb{N}$. For the sequence $\{a_i\}_{i=1}^\infty$ given in Subsection \ref{subsect2.2}, write
\begin{eqnarray}\label{200208t8}
\textbf{a}^{I,J}\triangleq\prod_{l=1}^{s}a_{i_l}^2\cdot \prod_{r=1}^{t}a_{j_r}^2.
\end{eqnarray}
Here, we agree with that $\prod\limits_{l=1}^{0}a_{i_l}^2=1$. Further, suppose that $I=(i_1,\cdots,i_s)$ and $J=(j_1,\cdots,j_t)$ are multi-indices with strictly increasing order, i.e., $i_1<\cdots<i_s$ and $j_1<\cdots<j_t$. We define a complex-valued function $\mathrm{d}z^I\wedge \mathrm{d}\overline{z}^J$ on $(\ell^p)^{s+t}$ by
\begin{eqnarray*}
(\mathrm{d}z^I\wedge \mathrm{d}\overline{z}^J)(\textbf{z}^1,\cdots,\textbf{z}^{s+t})\triangleq \frac{1}{\sqrt{(s+t)!}}\sum_{\sigma\in S_{s+t}}(-1)^{s(\sigma)}\cdot\prod_{k=1}^{s}z_{i_k}^{\sigma_k}\cdot\prod_{l=1}^{t}\overline{z_{j_l}^{\sigma_{s+l}}},\,\,
\end{eqnarray*}
where $\textbf{z}^l=(z_j^l)\in\ell^p,\,\,1\leq l\leq s+t$, $S_{s+t}$ is the permutation group of $\{1,\cdots,s+t\}$ and $s(\sigma)$ is the sign of $\sigma=(\sigma_1,\cdots,\sigma_s,\sigma_{s+1},\cdots,\sigma_{s+t})$, and we agree with that $0!=1$. Suppose that there exist $\{f_{I,J}\}\subset L^2(\ell^p,P)$ such that
\begin{eqnarray}\label{norm condition for st form}
\sum_{|I|=s,|J|=t}'2^{s+t}\textbf{a}^{I,J}\cdot\int_{\ell^p}|f_{I,J}|^2\,\mathrm{d}P<\infty,
\end{eqnarray}
here and henceforth, the sum $\sum'$ is taken only over strictly increasing multi-indices
$I=(i_1,\cdots,i_s)$ and $J=(j_1,\cdots,j_t)$. Now define
\begin{eqnarray}\label{st--form}
f(\textbf{z},\textbf{z}^1,\cdots,\textbf{z}^{s+t})\triangleq \sum_{|I|=s,|J|=t}'f_{I,J}(\textbf{z})(\mathrm{d}z^I\wedge \mathrm{d}\overline{z}^J)(\textbf{z}^1,\cdots,\textbf{z}^{s+t}),
\end{eqnarray}
where $\textbf{z},\,\textbf{z}^l=(z_j^l)\in\ell^p,\,\,1\leq l\leq s+t$. Then $f$ is an element of $L^2\bigg(\!(\ell^p)^{s+t+1},\prod\limits_{k=1}^{s+t+1}P\!\bigg)$ and the quantity in the left hand side of (\ref{norm condition for st form}) is the square of its norm.

Denote by $L^2_{(s,t)}(\ell^p)$ all the functions in $L^2\bigg((\ell^p)^{s+t+1},\prod\limits_{k=1}^{s+t+1}P\bigg)$ given as that of \eqref{st--form}. Each element in $L^2_{(s,t)}(\ell^p)$ is called an {\it $(s,t)$-form on $\ell^p$}. By Proposition \ref{Desity Theorem II}, $L^2_{(s,t)}(\ell^p)$ is a separable Hilbert space. By the definition above, it is easy to see that $L^2_{(s,t)}(\ell^p)$ is also a separable Hilbert space.

Choose $f\in L^2_{(s,t)}(\ell^p)$ in the form of \eqref{st--form}.
Suppose $j\in\mathbb{N}, I=(i_1,\cdots,i_s), J=(j_1,\cdots,j_t)$ and $K=(k_1,\cdots,k_{t+1})$ are multi-indices with strictly increasing order.
If $K\neq J\cup \{j\}$, set $\varepsilon_{j, J}^{K}=0$. If $K= J\cup \{j\}$, we denote the sign of the permutation taking $(j,j_1,\cdots,j_t)$ to $K$ by $\varepsilon_{j, J}^{K}$. If for each $I$ and $K$, there exists $g_{I,K}\in L^2(\ell^p,P)$ such that
\begin{eqnarray}\label{weak def of st form}
\int_{\ell^p}g_{I,K}\cdot\overline{\varphi}\,\mathrm{d}P=-\int_{\ell^p}\sum_{1\leq i<\infty}\sum_{|J|=t}'\varepsilon_{i, J}^{K}f_{I,J}\cdot\overline{\delta_i\varphi} \,\mathrm{d}P,\quad\forall\;\varphi\in \mathscr{C}_c^{\infty},
\end{eqnarray}
and
\begin{eqnarray*}
\sum_{|I|=s,|K|=t+1}'2^{s+t+1}\textbf{a}^{I,K}\cdot \int_{\ell^p}|g_{I,K}|^2\,\mathrm{d}P<\infty,
\end{eqnarray*}
then, we define
\begin{eqnarray}\label{d-dar-f-defi}
\overline{\partial}f(\textbf{z},\textbf{z}^1,\cdots,\textbf{z}^{s+t+1})\triangleq (-1)^s \sum_{|I|=s,|K|=t+1}'g_{I,K}(\textbf{z})(\mathrm{d}z^I\wedge \mathrm{d}\overline{z}^K)(\textbf{z}^1,\cdots,\textbf{z}^{s+t+1}),
\end{eqnarray}
where $\textbf{z}, \textbf{z}^l\in\ell^p,\,\,1\leq l\leq s+t+1$. Clearly,
$$\overline{\partial}f\in L^2_{(s,t+1)}(\ell^p).
 $$
Similarly to the proof of Lemma \ref{db is closable}, we can prove that the operator $\overline{\partial}$ given in the above is a densely defined, closed linear operator from $ L^2_{(s,t)}(\ell^p)$ to $L^2_{(s,t+1)}(\ell^p)$.

We have the following result:

\begin{lemma}\label{lem2.5.1}
Suppose that $f\in L^2_{(s,t)}(\ell^p)$ is in the form of \eqref{st--form} so that $\overline{\partial}f$ can be defined as that in \eqref{d-dar-f-defi}. Then,
\begin{eqnarray}\label{kernal-range-P}
\overline{\partial}(\overline{\partial}f)=0.
\end{eqnarray}
\end{lemma}

\begin{proof}
Suppose $I=(i_1,\cdots,i_s)$, $K=(k_1,\cdots,k_{t+1})$ and $M=(m_1,\cdots,m_{t+2})$ are multi-indices with strictly increasing order. We claim that, for each $j\in\mathbb{N}$ and $\varphi\in \mathscr{C}_c^{\infty}$,
\begin{eqnarray}\label{noncompact test}
\int_{\ell^p}g_{I,K}\cdot\overline{\delta_j\varphi}\,\mathrm{d}P
=-\int_{\ell^p}\sum_{1\leq i<\infty}\sum_{|J|=t}'\varepsilon_{i, J}^{K}f_{I,J} \cdot\overline{\delta_i(\delta_j\varphi)} \,\mathrm{d}P.
\end{eqnarray}
Here we should note that $\delta_j\varphi$ (defined by \eqref{delta-j}) may not lie in $\mathscr{C}_c^{\infty}$. Because of this, we choose a function $\varsigma\in C_c^\infty(\mathbb{R})$ so that $\varsigma\equiv 1$ on $[-1,1]$. For each $n\in\mathbb{N}$, write $\varphi_n\equiv \varsigma\big(\frac{\sum^{j}_{k=1}|z_k|^2}{n}\big)$. Then, $\delta_j(\varphi\varphi_n)\in \mathscr{C}_c^{\infty}$. Applying (\ref{weak def of st form}), we arrive at
\begin{eqnarray*}
\int_{\ell^p}g_{I,K}\cdot\overline{\delta_j(\varphi\varphi_n)}\,\mathrm{d}P
=-\int_{\ell^p}\sum_{1\leq i<\infty}\sum_{|J|=t}'\varepsilon_{i, J}^{K}f_{I,J}\cdot\overline{\delta_i(\delta_j(\varphi\varphi_n))} \,\mathrm{d}P.
\end{eqnarray*}
Letting $n\to\infty$ in the above, we obtain \eqref{noncompact test}.

Thanks to \eqref{noncompact test}, it follows that
\begin{eqnarray*}
&&-\int_{\ell^p}\sum_{1\leq j<\infty}\sum_{|K|=t+1}'\varepsilon_{j, K}^{M}g_{I,K}\cdot\overline{\delta_j\varphi}\,\mathrm{d}P\\
&=&\sum_{1\leq i,j<\infty}\sum_{|K|=t+1}'\sum_{|J|=t}'\varepsilon_{i, J}^{K}\cdot\varepsilon_{j, K}^{M}\cdot\int_{\ell^p}f_{I,J}\cdot\overline{\delta_i(\delta_j\varphi)}\,\mathrm{d}P\\
&=&\sum_{1\leq i,j<\infty}\sum_{|K|=t+1}'\sum_{|J|=t}'\varepsilon_{j,i, J}^{j,K}\cdot\varepsilon_{j, K}^{M}\cdot\int_{\ell^p}f_{I,J}\cdot\overline{\delta_i(\delta_j\varphi)}\,\mathrm{d}P\\
&=&\sum_{1\leq i,j<\infty}\sum_{|J|=t}'\varepsilon_{j,i, J}^{M}\cdot\int_{\ell^p}f_{I,J}\cdot\overline{\delta_i(\delta_j\varphi)}\,\mathrm{d}P\\
&=&-\sum_{1\leq i,j<\infty}\sum_{|J|=t}'\varepsilon_{i,j, J}^{M}\cdot\int_{\ell^p}f_{I,J}\cdot\overline{\delta_i(\delta_j\varphi)}\,\mathrm{d}P\\
&=&-\sum_{1\leq i,j<\infty}\sum_{|J|=t}'\varepsilon_{i,j, J}^{M}\cdot\int_{\ell^p}f_{I,J}\cdot\overline{\delta_j(\delta_i\varphi)}\,\mathrm{d}P\\
&=&-\sum_{1\leq i,j<\infty}\sum_{|J|=t}'\varepsilon_{j,i, J}^{M}\cdot\int_{\ell^p}f_{I,J}\cdot\overline{\delta_i(\delta_j\varphi)}\,\mathrm{d}P,
\end{eqnarray*}
where the fourth equality follows from the fact that $\varepsilon_{j,i, J}^{M}=-\varepsilon_{i,j, J}^{M}$. The above equalities imply that
\begin{eqnarray*}
-\int_{\ell^p}\sum_{1\leq j<\infty}\sum_{|K|=t+1}'\varepsilon_{j, K}^{M}g_{I,K}\cdot\overline{\delta_j\varphi}\,\mathrm{d}P=0.
\end{eqnarray*}
The arbitrariness of $I$ and $M$ implies the desired property \eqref{kernal-range-P}. This complete the proof of Lemma \ref{lem2.5.1}.
\end{proof}

\section{$L^2$ estimate for smooth $(s,t+1)$-forms on $\ell^p$}
Suppose that $s,t$ are non-negative integers. Denote the $\overline{\partial }$ operator from $L^2_{(s,t)}(\ell^p)$ into $L^2_{(s,t+1)}(\ell^p)$ by $T$ and the $\overline{\partial }$ operator from $L^2_{(s,t+1)}(\ell^p)$ into $L^2_{(s,t+2)}(\ell^p)$ by $S$. Clearly, Lemmas \ref{db is closable} and \ref{lem2.5.1} imply that $R_T\subset N_S$ and $N_S$ is a closed subspace of $L^2_{(s,t+1)}(\ell^p)$.  We will prove that $R_T=N_S$ in Section \ref{Sec5}.  Write
\begin{eqnarray*}
\mathscr{D}_{(s,t)}\triangleq \bigcup_{n=1}^{\infty}\Bigg\{f= \sum_{|I|=s,|J|=t,\,\max(I\cup J)\leq n}'f_{I,J}\,\mathrm{d}z^I\wedge \mathrm{d}\overline{z}^J:\;\,\,f_{I,J}\in \mathscr{C}_c^{\infty}\Bigg\}.
\end{eqnarray*}
Then $\mathscr{D}_{(s,t)}$ is dense in $L^2_{(s,t)}(\ell^p)$.

We have the following $L^2$ estimate for $(s,t+1)$-forms in $\mathscr{D}_{(s,t+1)}$:

\begin{theorem}\label{200207t1}
It holds that
\begin{eqnarray}\label{202002071}
||f||_{L^2_{(s,t+1)}(\ell^p)}^2\leq ||T^*f||_{L^2_{(s,t)}(\ell^p)}^2
+||Sf||_{L^2_{(s,t+2)}(\ell^p)}^2,\quad \forall\;f\in  \mathscr{D}_{(s,t+1)}.
\end{eqnarray}
\end{theorem}

\begin{proof}We borrow some idea from H\"omander \cite{Hor65}. Before proceeding, we need two simple equalities (Recall \eqref{delta-j} for $\delta_j$):
\begin{eqnarray}
\int_{\ell^p}\frac{\partial \varphi}{\partial \overline{z_j}}\cdot\overline{\psi} \,\mathrm{d}P
=-\int_{\ell^p}\varphi\cdot\overline{\delta_j\psi}\,\mathrm{d}P,\,\,\quad\forall\,j\in\mathbb{N},\,\,\,\varphi,\psi\in\mathscr{C}_c^{\infty},\label{22.1.8'}
\end{eqnarray}
and
\begin{eqnarray}
\bigg(\delta_k\frac{\partial  }{\partial \overline{z_j}}-\frac{\partial  }{\partial \overline{z_j}}\delta_k\bigg)w=\frac{w}{2a_k^2}\frac{\partial\overline{z_k}}{\partial \overline{z_j}} ,\,\,\quad\forall\,k,j\in\mathbb{N},\,\,\,w\in \mathscr{C}_c^{\infty}.\label{22.1.8}
\end{eqnarray}

Suppose that $f\in \mathscr{D}_{(s,t+1)}$ and $u\in D_{T}(\subset L^2_{(s,t)}(\ell^p))$ are as follows:
\begin{eqnarray*}
f= \sum_{|I|=s,|J|=t+1}'f_{I,J}\,\mathrm{d}z^I\wedge \mathrm{d}\overline{z}^J,
\quad\,\,u= \sum_{|I|=s,|K|=t}'u_{I,K}\,\mathrm{d}z^I\wedge \mathrm{d}\overline{z}^K
\end{eqnarray*}
and $Tu= (-1)^s\sum\limits_{|I|=s,|J|=t+1}'g_{I,J}\,\mathrm{d}z^I\wedge \mathrm{d}\overline{z}^J$ such that
\begin{eqnarray}\label{2weak def of st form on Cn}
\int_{\ell^p}g_{I,J}\overline{\varphi}\,\mathrm{d}P=-\int_{\ell^p}\sum_{1\leq j<\infty}\sum_{|K|=t}'\varepsilon_{j, K}^{J}u_{I,K}\cdot\overline{\delta_j\varphi} \,\mathrm{d}P,\quad\,\forall\;\varphi\in \mathscr{C}_c^{\infty}.
\end{eqnarray}
Then,
\begin{eqnarray*}
&&(f,T u)_{L^2_{(s,t+1)}(\ell^p)}\\
&=&(-1)^s\cdot 2^{s+t+1}\cdot\sum_{|I|=s,|J|=t+1}' \textbf{a}^{I,J} \cdot\int_{\ell^p}f_{I,J}\cdot\overline{g_{I,J}}\,\mathrm{d}P\\
&=&(-1)^{s-1}\cdot 2^{s+t+1}\cdot\sum_{|I|=s,|J|=t+1}' \textbf{a}^{I,J}\cdot \int_{\ell^p}\sum_{1\leq j<\infty}\sum_{|K|=t}'\varepsilon_{j, K}^{J}\cdot\delta_j(f_{I,J})\cdot\overline{u_{I,K}} \,\mathrm{d}P\\
&=&(-1)^{s-1}\cdot 2^{s+t}\cdot\sum_{|I|=s,|K|=t}'\textbf{a}^{I,K}\cdot\int_{\ell^p} \sum_{1\leq j<\infty}\sum_{|J|=t+1}'2a_j^2\cdot\varepsilon_{j,K}^{J}\cdot \delta_j(f_{I,J})\cdot\overline{ u_{I,K}}\,\mathrm{d}P.
\end{eqnarray*}
Note that if $J\neq \{j\}\cup K$ then $\varepsilon_{j,K}^{J}=0$. Therefore, if $j\in K$ we let $f_{I,jK}\triangleq0$ and if $j\notin K$, there exists a unique multi-index $J$ with strictly order such that $|J|=t+1$ and $J= \{j\}\cup K$, we let $f_{I,jK}\triangleq\varepsilon_{j,K}^{J}\cdot f_{I,J}$. Therefore,
\begin{eqnarray*}
(f,T u)_{L^2_{(s,t+1)}(\ell^p)}
&=&(-1)^{s-1}\cdot 2^{s+t}\cdot\sum_{|I|=s,|K|=t}'\textbf{a}^{I,K}\cdot\int_{\ell^p} \sum_{1\leq j<\infty} 2a_j^2\delta_j(f_{I,jK})\cdot\overline{ u_{I,K}}\,\mathrm{d}P.
\end{eqnarray*}
Then $f\in D_{T^*}$ and
\begin{eqnarray*}
T^*f&=&(-1)^{s-1}\sum_{|I|=s,|K|=t}'\sum_{1\leq j<\infty}2a_j^2\delta_j(f_{I,jK})\mathrm{d}z^I\wedge \mathrm{d}\overline{z}^K.
\end{eqnarray*}
Thus
\begin{eqnarray*}
||T^*f||_{L^2_{(s,t)}(\ell^p)}^2
\!\!&=&2^{s+t+2}\sum_{|I|=s,|K|=t}'\textbf{a}^{I,K}\int_{\ell^p} \sum_{1\leq j<\infty}a_j^2\delta_j(f_{I,jK})\cdot\overline{ \sum_{1\leq i<\infty}a_i^2\delta_i(f_{I,iK})}\,\mathrm{d}P\\
&=&2^{s+t+2}\sum_{|I|=s,|K|=t}'\textbf{a}^{I,K} \sum_{1\leq i,j<\infty}\int_{\ell^p}a_j^2\delta_j(f_{I,jK})\cdot\overline{ a_i^2\delta_i(f_{I,iK})}\,\mathrm{d}P.
\end{eqnarray*}

On the other hand, combining (\ref{weak def of st form}), (\ref{d-dar-f-defi}) and (\ref{22.1.8'}) gives
\begin{eqnarray}\label{formula of d-bar}
Sf=(-1)^s\sum_{|I|=s,|M|=t+2}'\sum_{1\leq j<\infty}\sum_{|J|=t+1}'\varepsilon_{j,J}^{M}\cdot\frac{\partial f_{I,J}}{\partial \overline{z_j}}\,\mathrm{d}z^I\wedge \mathrm{d}\overline{z}^M.
\end{eqnarray}
Hence,
\begin{eqnarray*}
&&||Sf||_{L^2_{(s,t+2)}(\ell^p)}^2\\
&=&2^{s+t+2}\cdot\sum_{|I|=s,|M|=t+2}'\textbf{a}^{I,M}\cdot\int_{\ell^p}\bigg(\sum_{1\leq j<\infty}\sum_{|K|=t+1}'\varepsilon_{j,K}^{M}\cdot\frac{\partial f_{I,K}}{\partial \overline{z_j}}\\
&&\qquad\qquad\qquad\qquad\qquad\qquad\qquad\qquad\qquad\cdot\;\overline{\sum_{1\leq i<\infty}\sum_{|L|=t+1}'\varepsilon_{i,L}^{M}\cdot\frac{\partial f_{I,L}}{\partial \overline{z_i}}}\;\bigg)\,\mathrm{d}P\\
&=&2^{s+t+2}\sum_{|I|=s,|M|=t+2}'\sum_{1\leq i, j<\infty}\sum_{|K|=t+1}'\sum_{|L|=t+1}'\textbf{a}^{I,M}\int_{\ell^p}\varepsilon_{j,K}^{M}\varepsilon_{i,L}^{M}\frac{\partial f_{I,K}}{\partial \overline{z_j}}\cdot\overline{\frac{\partial f_{I,L}}{\partial \overline{z_i}}}\,\mathrm{d}P.
\end{eqnarray*}
Note that $\varepsilon_{j,K}^{M}\cdot\varepsilon_{i,L}^{M}\neq 0$ if and only if $j\notin K,\,i\notin L$ and $M=\{j\}\cup K=\{i\}\cup L$, in this case $\varepsilon_{j,K}^{M}\cdot\varepsilon_{i,L}^{M}=\varepsilon_{i,L}^{j,K}$. Thus we have
\begin{equation}\label{norm of Sf}
\begin{array}{ll}
\displaystyle||Sf||_{L^2_{(s,t+2)}(\ell^p)}^2\\[3mm]\displaystyle
=2^{s+t+2}\cdot\sum_{|I|=s,|K|=t+1,|L|=t+1}'\sum_{1\leq i,j<\infty}\textbf{a}^{I,K}\cdot a_j^2\cdot\int_{\ell^p}\varepsilon_{i,L}^{j,K}\frac{\partial f_{I,K}}{\partial \overline{z_j}}\cdot\overline{\frac{\partial f_{I,L}}{\partial \overline{z_i}}}\,\mathrm{d}P.
\end{array}
\end{equation}
If $i=j$ then, $\varepsilon_{i,L}^{j,K}\neq 0$ if and only if $i\notin K$ and $K=L$. Hence, the corresponding part of (\ref{norm of Sf}) is
\begin{eqnarray*}
2^{s+t+2}\cdot\sum_{|I|=s,|K|=t+1}'\sum_{ i\notin K}\textbf{a}^{I,K}\cdot a_i^2\cdot\int_{\ell^p}\bigg|\frac{\partial f_{I,K}}{\partial \overline{z_i}}\bigg|^2\,\mathrm{d}P.
\end{eqnarray*}
Also, if $i\neq j$ then, $\varepsilon_{i,L}^{j,K}\neq 0$ if and only if $i\in K,\,j\in L$ and $i\notin L,\,j\notin K$. Then there exists multi-index $J$ with strictly increasing order such that $|J|=t$ and $J\cup\{i,j\}=\{i\}\cup L=\{j\}\cup K$. Since $\varepsilon_{i,L}^{j,K}=\varepsilon_{i,L}^{i,j,J}\cdot\varepsilon_{i,j,J}^{j,i,J}\cdot\varepsilon_{j,i,J}^{j,K}
=-\varepsilon_{L}^{j,J}\cdot\varepsilon_{i,J}^{K}$, the rest part of (\ref{norm of Sf}) is
\begin{eqnarray*}
&&-2^{s+t+2}\sum_{|I|=s,|K|=t+1,|L|=t+1,|J|=t}'\sum_{i\neq j,\,1\leq i,j<\infty}\textbf{a}^{I,K} a_j^2\int_{\ell^p}\varepsilon_{L}^{j,J}\varepsilon_{i,J}^{K}\frac{\partial f_{I,K}}{\partial \overline{z_j}}\cdot\overline{\frac{\partial f_{I,L}}{\partial \overline{z_i}}}\,\mathrm{d}P\\
&=&-2^{s+t+2}\sum_{|I|=s,|K|=t+1,|L|=t+1,|J|=t}'\sum_{i\neq j,\,1\leq i,j<\infty}\textbf{a}^{I,K} a_j^2\int_{\ell^p}\varepsilon_{i,J}^{K}\frac{\partial f_{I,K}}{\partial \overline{z_j}}\cdot\overline{\varepsilon_{L}^{j,J}\frac{\partial f_{I,L}}{\partial \overline{z_i}}}\,\mathrm{d}P\\
&=&-2^{s+t+2}\sum_{|I|=s,|J|=t}'\sum_{ i\neq j}\textbf{a}^{I,J} a_j^2 a_i^2\int_{\ell^p}\frac{\partial f_{I,iJ}}{\partial \overline{z_j}}\cdot\overline{\frac{\partial f_{I,jJ}}{\partial \overline{z_i}}}\,\mathrm{d}P.
\end{eqnarray*}
Therefore, we have
\begin{eqnarray*}
&&||T^*f||_{L^2_{(s,t)}(\ell^p)}^2
+
||Sf||_{L^2_{(s,t+2)}(\ell^p)}^2\\
&=&2^{s+t+2}\cdot\sum_{|I|=s,|J|=t}'\sum_{1\leq i,j<\infty} \textbf{a}^{I,J}\cdot a_j^2\cdot a_i^2\cdot\int_{\Omega}\delta_j(f_{I,jJ})\cdot\overline{\delta_i(f_{I,iJ})}\,\mathrm{d}P\\
&&+2^{s+t+2}\cdot\sum_{|I|=s,|K|=t+1}'\sum_{ i\notin K}\textbf{a}^{I,K}\cdot a_i^2\cdot\int_{\ell^p}\bigg|\frac{\partial f_{I,K}}{\partial \overline{z_i}}\bigg|^2\,\mathrm{d}P\\
&&-2^{s+t+2}\cdot\sum_{|I|=s,|J|=t}'\sum_{ i\neq j}\textbf{a}^{I,J}\cdot a_j^2\cdot a_i^2\cdot\int_{\ell^p}\frac{\partial f_{I,iJ}}{\partial \overline{z_j}}\cdot\overline{\frac{\partial f_{I,jJ}}{\partial \overline{z_i}}}\,\mathrm{d}P\\
&=&2^{s+t+2}\cdot\sum_{|I|=s,|J|=t}'\sum_{1\leq i,j<\infty} \textbf{a}^{I,J}\cdot a_j^2\cdot a_i^2\cdot\int_{\Omega}\delta_j(f_{I,jJ})\cdot\overline{ \delta_i(f_{I,iJ})}\,\mathrm{d}P\\
&&+2^{s+t+2}\sum_{|I|=s,|K|=t+1}'\sum_{ 1\leq i<\infty}\textbf{a}^{I,K}\cdot a_i^2\cdot\int_{\ell^p}\bigg|\frac{\partial f_{I,K}}{\partial \overline{z_i}}\bigg|^2\,\mathrm{d}P\\
&&-2^{s+t+2}\cdot\sum_{|I|=s,|J|=t}'\sum_{1\leq i,j<\infty}\textbf{a}^{I,J}\cdot a_j^2\cdot a_i^2\cdot\int_{\ell^p}\frac{\partial f_{I,iJ}}{\partial \overline{z_j}}\cdot\overline{\frac{\partial f_{I,jJ}}{\partial \overline{z_i}}}\,\mathrm{d}P\\
&=&2^{s+t+2}\cdot\sum_{|I|=s,|J|=t}' \sum_{1\leq i,j<\infty}\textbf{a}^{I,J}\cdot a_j^2\cdot a_i^2\cdot\int_{\ell^p}\Bigg(\delta_j(f_{I,jJ})\cdot\overline{ \delta_i(f_{I,iJ})}\\
&&\qquad\qquad\qquad\qquad\qquad\qquad\qquad\qquad\qquad\qquad\qquad-\frac{\partial f_{I,jJ}}{\partial \overline{z_i}}\cdot\overline{\frac{\partial f_{I,iJ}}{\partial \overline{z_j}}}\Bigg)\,\mathrm{d}P\\
&&+2^{s+t+2}\sum_{|I|=s,|K|=t+1}'\sum_{ 1\leq i<\infty}\textbf{a}^{I,K}\cdot a_i^2\cdot\int_{\ell^p}\bigg|\frac{\partial f_{I,K}}{\partial \overline{z_i}}\bigg|^2\,\mathrm{d}P,
\end{eqnarray*}
where the second equality follows from
\begin{eqnarray*}
2^{s+t+2}\sum_{|I|=s,|K|=t+1}'\sum_{ i\in K}\textbf{a}^{I,K}\cdot a_i^2\cdot\int_{\ell^p}\bigg|\frac{\partial f_{I,K}}{\partial \overline{z_i}}\bigg|^2\,\mathrm{d}P\\
=2^{s+t+2}\sum_{|I|=s,|J|=t}'\sum_{1\leq i <\infty}\textbf{a}^{I,J}\cdot  a_i^4\cdot\int_{\ell^p}\bigg|\frac{\partial f_{I,iJ}}{\partial \overline{z_i}}\bigg|^2\,\mathrm{d}P,
\end{eqnarray*}
and the third equality follows from
\begin{eqnarray*}
&&-2^{s+t+2}\cdot\sum_{|I|=s,|J|=t}'\sum_{1\leq i,j<\infty}\textbf{a}^{I,J}\cdot a_j^2\cdot a_i^2\cdot\int_{\ell^p}\frac{\partial f_{I,iJ}}{\partial \overline{z_j}}\cdot\overline{\frac{\partial f_{I,jJ}}{\partial \overline{z_i}}}\,\mathrm{d}P\\
&=&-2^{s+t+2}\cdot\sum_{|I|=s,|J|=t}' \sum_{1\leq i,j<\infty}\textbf{a}^{I,J}\cdot a_j^2\cdot a_i^2\cdot\int_{\ell^p}\frac{\partial f_{I,jJ}}{\partial \overline{z_i}}\cdot\overline{\frac{\partial f_{I,iJ}}{\partial \overline{z_j}}}\,\mathrm{d}P.
\end{eqnarray*}
In view of $(\ref{22.1.8'})$ and $(\ref{22.1.8})$, we conclude that
\begin{eqnarray*}
&&||T^*f||_{L^2_{(s,t)}(\ell^p)}^2
+
||Sf||_{L^2_{(s,t+2)}(\ell^p)}^2\\
&=&2^{s+t+2}\sum_{|I|=s,|J|=t}' \sum_{1\leq i,j<\infty}\textbf{a}^{I,J}\cdot a_j^2\cdot a_i^2\cdot\int_{\ell^p}\frac{1}{2a_j^2}\cdot \frac{\partial\overline{z_j}}{\partial \overline{z_i}}\cdot f_{I,jJ}\cdot\overline{f_{I,iJ}}\,\mathrm{d}P\\
&&+2^{s+t+2}\sum_{|I|=s,|K|=t+1}'\sum_{ 1\leq i<\infty}\textbf{a}^{I,K}\cdot a_i^2\cdot\int_{\ell^p}\bigg|\frac{\partial f_{I,K}}{\partial \overline{z_i}}\bigg|^2\,\mathrm{d}P\\
&=&2^{s+t+1}\sum_{|I|=s,|J|=t}' \sum_{1\leq i<\infty}\textbf{a}^{I,J}\cdot a_i^2\cdot\int_{\ell^p}|f_{I,iJ}|^2\,\mathrm{d}P\\
&&+2^{s+t+1}\sum_{|I|=s,|K|=t+1}'\textbf{a}^{I,K}\cdot ||\overline{\partial}f_{I,K}||_{L^2_{(0,1)}(\ell^p)}^2\\
&=&(t+1)\cdot||f||_{L^2_{(s,t+1)}(\ell^p)}^2+2^{s+t+1}\cdot\sum_{|I|=s,|K|=t+1}'\textbf{a}^{I,K}\cdot ||\overline{\partial}f_{I,K}||_{L^2_{(0,1)}(\ell^p)}^2\\
&\geq & ||f||_{L^2_{(s,t+1)}(\ell^p)}^2.
\end{eqnarray*}
which gives the desired $L^2$ estimate \eqref{202002071}. This completes the proof of Theorem \ref{200207t1}.
\end{proof}

\begin{remark}\label{200209rem1}
In order to solve the $\overline{\partial}$ equation \eqref{d-bar equation} on $\ell^p$, or equivalently, to prove that $R_{T}=N_{S}$, by the conclusion 1) in Corollary \ref{estimation lemma} we need to improve the estimate \eqref{202002071} in Theorem \ref{200207t1} as follows:
\begin{eqnarray}\label{200209tt1}
||f||_{L^2_{(s,t+1)}(\ell^p)}^2\leq ||T^*f||_{L^2_{(s,t)}(\ell^p)}^2+||Sf||_{L^2_{(s,t+2)}(\ell^p)}^2,\quad\forall\,f\in D_{S}\cap D_{T^*}.
\end{eqnarray}
In the case of finite dimensions, a similar improvement was achieved in the work \cite{Hor65} by means of the classical Friedrichs mollifier technique. Unfortunately, this mollifier technique uses essentially the translation invariance of Lebesgue measure, and therefore it cannot be employed directly to prove \eqref{200209tt1}. Also, at this moment, it is unclear for us whether the estimate \eqref{200209tt1} holds or not.
\end{remark}

\section{Solving the $\overline{\partial}$ equation uniformly on $\mathbb{C}^n$}

The main goal in this section is to solve the $\overline{\partial}$ equation \eqref{d-bar equation} uniformly in finite dimensions.

For each $n\in \mathbb{N}$ and nonnegative integers $s$ and $t$ which are less than or equal to $n$, very similarly to that for $L^2_{(s,t)}(\ell^p)$, we may define $L^2_{(s,t)}(\mathbb{C}^n)$ and their norms (and hence we omit the details).
By a similar argument, we can define the $\overline{\partial}$ operators on $\mathbb{C}^n$. The differences between this case and that on $\ell^p$  are the following:

\bigno
\textbf{1.} We use the measure $\mathcal {N}^n$ instead of $P$;\\
\textbf{2.} We choose $\mathscr{C}_c^{\infty}(\mathbb{C}^n)$ instead of $\mathscr {C}_c^{\infty}$ as the test functions;\\
\textbf{3.} We use the notation $L^2_{(s,t)}(\mathbb{C}^n)$ instead of $L^2_{(s,t)}(\ell^p)$ to denote all the $(s,t)$-forms on $\mathbb{C}^n$, where $0\le s,t\le n$. Similar to Remark \ref{identify remark}, $L^2_{(s,t)}(\mathbb{C}^n)$ can be naturally identified as a closed subspace of $L^2_{(s,t)}(\ell^p)$;\\
\textbf{4.} Suppose $I=(i_1,\cdots,i_s)$ and $J=(j_1,\cdots,j_t)$ are multi-indices between $1$ and $n$, with strictly increasing order. Write $\textbf{a}^{I,J}$ as that in \eqref{200208t8}.\\
\textbf{5.} We denote by $T_n$ the $\overline{\partial }$ operator from $L^2_{(s,t)}(\mathbb{C}^n)$ into $L^2_{(s,t+1)}(\mathbb{C}^n)$  and by $S_n$ the $\overline{\partial }$ operator from $L^2_{(s,t+1)}(\mathbb{C}^n)$ into $L^2_{(s,t+2)}(\mathbb{C}^n)$. Similar arguments implies that $N_{S_n}$ is a closed subspace of $L^2_{(s,t+1)}(\mathbb{C}^n)$ and $R_{T_n}\subset N_{S_n}$;\\
\textbf{6.} Write
\begin{eqnarray*}
\mathscr{D}_{(s,t)}(\mathbb{C}^n)\triangleq \Bigg\{f= \sum_{|I|=s,|J|=t}'f_{I,J}\,\mathrm{d}z^I\wedge \mathrm{d}\overline{z}^J:\,\,f_{I,J}\in \mathscr{C}_c^{\infty}(\mathbb{C}^n)\Bigg\}.
\end{eqnarray*}
Then $\mathscr{D}_{(s,t)}(\mathbb{C}^n)$ is dense in $L^2_{(s,t)}(\mathbb{C}^n)$.\\
\textbf{7.} In this section, the sum $\sum'$ is taken only over strictly increasing multi-indices between $1$ and $n$.

We shall prove that $R_{T_n}=N_{S_n}$ and in particular solve the $\overline{\partial}$ equation on $\mathbb{C}^n$ in a manner that the solutions can be used to built up the desired solution to a similar $\overline{\partial}$ equation but on $\ell^p$. The key is to establish a suitable dimension-free {\it a priori} estimate.

First of all, quite similarly to Theorem \ref{200207t1}, we have the following $L^2$ estimate for $(s,t+1)$-forms in $\mathscr{D}_{(s,t)}(\mathbb{C}^n)$:

\begin{theorem}\label{200208t1}
 It holds that
\begin{eqnarray}\label{200208e6}
||f||_{L^2_{(s,t+1)}(\mathbb{C}^n)}^2\leq ||T_n^*f||_{L^2_{(s,t)}(\mathbb{C}^n)}^2+||S_nf||_{L^2_{(s,t+2)}(\mathbb{C}^n)}^2,\quad\forall\,f\in \mathscr{D}_{(s,t)}(\mathbb{C}^n).
\end{eqnarray}
\end{theorem}

\begin{proof}
The proof is actually very close to that of \cite[Proposition 2.1.2]{Hor65} and Theorem \ref{200207t1} but, for the readers' convenience and also for completeness, we shall give below the details.

Suppose $f\in \mathscr{D}_{(s,t+1)}(\mathbb{C}^n)$ and $u\in D_{T_n}$ with
\begin{eqnarray*}
f= \sum_{|I|=s,|J|=t+1}'f_{I,J}\,\mathrm{d}z^I\wedge \mathrm{d}\overline{z}^J,
\,\,\,u= \sum_{|I|=s,|K|=t}'u_{I,K}\,\mathrm{d}z^I\wedge \mathrm{d}\overline{z}^K
\end{eqnarray*}
and $T_nu= (-1)^s\sum\limits_{|I|=s,|J|=t+1}'g_{I,J}\,\mathrm{d}z^I\wedge \mathrm{d}\overline{z}^J$ such that
\begin{eqnarray}\label{weak def of st form on Cn}
\int_{\mathbb{C}^n}g_{I,J}\overline{\varphi}\,\mathrm{d}\mathcal {N}^n=-\int_{\mathbb{C}^n}\sum_{1\leq j\leq n}\sum_{|K|=t}'\varepsilon_{j, K}^{J}u_{I,K}\cdot \overline{\delta_j\varphi} \,\mathrm{d}\mathcal {N}^n,\quad\forall\,\varphi\in \mathscr{C}_c^{\infty}(\mathbb{C}^n).
\end{eqnarray}
Then
\begin{eqnarray*}
&&(f,T_n u)_{L^2_{(s,t+1)}(\mathbb{C}^n)}\\
&=&(-1)^s\cdot 2^{s+t+1}\cdot\sum_{|I|=s,|J|=t+1}' \textbf{a}^{I,J} \int_{\mathbb{C}^n}f_{I,J}\cdot\overline{g_{I,J}}\,\mathrm{d}\mathcal {N}^n\\
&=&(-1)^{s-1}\cdot 2^{s+t+1}\cdot\sum_{|I|=s,|J|=t+1}' \textbf{a}^{I,J} \int_{\mathbb{C}^n}\sum_{1\leq j\leq n}\sum_{|K|=t}'\varepsilon_{j, K}^{J}\cdot\delta_j(f_{I,J})\cdot\overline{u_{I,K}} \,\mathrm{d}\mathcal {N}^n\\
&=&(-1)^{s-1}\cdot 2^{s+t}\cdot\sum_{|I|=s,|K|=t}'\textbf{a}^{I,K}\int_{\mathbb{C}^n}\sum_{1\leq j\leq n}\sum_{|J|=t+1}'2a_j^2\cdot\varepsilon_{j,K}^{J}\cdot \delta_j(f_{I,J})\cdot\overline{ u_{I,K}}\,\mathrm{d}\mathcal {N}^n.
\end{eqnarray*}
Note that if $J\neq \{j\}\cup K$ then $\varepsilon_{j,K}^{J}=0$. Therefore, if $j\in K$ we let $f_{I,jK}\triangleq0$ and if $j\notin K$, there exists a unique multi-index $J$ with strictly order such that $|J|=t+1$ and $J= \{j\}\cup K$, we let $f_{I,jK}\triangleq\varepsilon_{j,K}^{J}\cdot f_{I,J}$. Therefore,
\begin{eqnarray*}
(f,T_n u)_{L^2_{(s,t+1)}(\mathbb{C}^n)}
=(-1)^{s-1}\cdot 2^{s+t}\cdot\sum_{|I|=s,|K|=t}'\textbf{a}^{I,K}\int_{\mathbb{C}^n}\sum_{1\leq j\leq n} 2a_j^2\delta_j(f_{I,jK})\cdot\overline{ u_{I,K}}\,\mathrm{d}\mathcal {N}^n.
\end{eqnarray*}
This implies that $f\in D_{T_n^*}$ and
\begin{eqnarray*}
T_n^*f&=&(-1)^{s-1}\sum_{|I|=s,|K|=t}'\sum_{1\leq j\leq n}2a_j^2\delta_j(f_{I,jK})\,\mathrm{d}z^I\wedge \mathrm{d}\overline{z}^K.
\end{eqnarray*}
Thus
\begin{eqnarray*}
||T_n^*f||_{L^2_{(s,t)}(\mathbb{C}^n)}^2
&=&2^{s+t+2}\sum_{|I|=s,|K|=t}'\textbf{a}^{I,K}\int_{\mathbb{C}^n}\sum_{1\leq j\leq n}a_j^2\delta_j(f_{I,jK})\cdot\overline{\sum_{1\leq i\leq n}a_i^2\delta_i(f_{I,iK})}\,\mathrm{d}\mathcal {N}^n\\
&=&2^{s+t+2}\sum_{|I|=s,|K|=t}'\textbf{a}^{I,K}\sum_{1\leq i,j\leq n}\int_{\mathbb{C}^n}a_j^2\delta_j(f_{I,jK})\cdot\overline{ a_i^2\delta_i(f_{I,iK})}\,\mathrm{d}\mathcal {N}^n.
\end{eqnarray*}

Now, similarly to (\ref{formula of d-bar}) we have
\begin{eqnarray*}
S_nf=(-1)^s\sum_{|I|=s,|M|=t+2}'\sum_{1\leq j\leq n}\sum_{|J|=t+1}'\varepsilon_{j,J}^{M}\frac{\partial f_{I,J}}{\partial \overline{z_j}}\,\mathrm{d}z^I\wedge \mathrm{d}\overline{z}^M,
\end{eqnarray*}
which implies that
\begin{eqnarray*}
&&||S_nf||_{L^2_{(s,t+2)}(\mathbb{C}^n)}^2\\
&=&2^{s+t+2}\sum_{|I|=s,|M|=t+2}'\textbf{a}^{I,M}\int_{\mathbb{C}^n}\sum_{1\leq j\leq n}\sum_{|K|=t+1}'\varepsilon_{j,K}^{M}\frac{\partial f_{I,K}}{\partial \overline{z_j}}\cdot\overline{\sum_{1\leq i\leq n}\sum_{|L|=t+1}'\varepsilon_{i,L}^{M}\frac{\partial f_{I,L}}{\partial \overline{z_i}}}\,\mathrm{d}\mathcal {N}^n\\
&=&2^{s+t+2}\sum_{|I|=s,|M|=t+2}'\sum_{1\leq i, j\leq n}\sum_{|K|=t+1}'\sum_{|L|=t+1}'\textbf{a}^{I,M}\int_{\mathbb{C}^n}\varepsilon_{j,K}^{M}\varepsilon_{i,L}^{M}\frac{\partial f_{I,K}}{\partial \overline{z_j}}\cdot\overline{\frac{\partial f_{I,L}}{\partial \overline{z_i}}}\,\mathrm{d}\mathcal {N}^n.
\end{eqnarray*}
Note that $\varepsilon_{j,K}^{M}\cdot\varepsilon_{i,L}^{M}\neq 0$ if and only if $j\notin K,\,i\notin L$ and $M=\{j\}\cup K=\{i\}\cup L$, in this case $\varepsilon_{j,K}^{M}\varepsilon_{i,L}^{M}=\varepsilon_{i,L}^{j,K}$. Thus
\begin{equation}\label{200228e5}
\begin{array}{ll}
\displaystyle ||S_nf||_{L^2_{(s,t+2)}(\mathbb{C}^n)}^2\\[3mm]
\displaystyle=2^{s+t+2}\sum_{|I|=s,|K|=t+1,|L|=t+1}'\sum_{1\leq i,j\leq n}\textbf{a}^{I,K} a_j^2\int_{\mathbb{C}^n}\varepsilon_{i,L}^{j,K}\frac{\partial f_{I,K}}{\partial \overline{z_j}}\cdot\overline{\frac{\partial f_{I,L}}{\partial \overline{z_i}}}\,\mathrm{d}\mathcal {N}^n.
\end{array}
\end{equation}
Clearly, if $i=j$ then, $\varepsilon_{i,L}^{j,K}\neq 0$ if and only if $i\notin K$ and $K=L$. Hence the corresponding part in \eqref{200228e5} is
\begin{eqnarray*}
2^{s+t+2}\sum_{|I|=s,|K|=t+1}'\sum_{ i\notin K}\textbf{a}^{I,K} a_i^2\int_{\mathbb{C}^n}\bigg|\frac{\partial f_{I,K}}{\partial \overline{z_i}}\bigg|^2\,\mathrm{d}\mathcal {N}^n.
\end{eqnarray*}
If $i\neq j$ then, $\varepsilon_{i,L}^{j,K}\neq 0$ if and only if $i\in K,\,j\in L$ and $i\notin L,\,j\notin K$. Therefore, there exists $|J|=t$ with strictly increasing index such that $J\cup\{i,j\}=\{i\}\cup L=\{j\}\cup K$. Since $\varepsilon_{i,L}^{j,K}=\varepsilon_{i,L}^{i,j,J}\varepsilon_{i,j,J}^{j,i,J}\varepsilon_{j,i,J}^{j,K}
=-\varepsilon_{L}^{j,J}\varepsilon_{i,J}^{K}$, the rest part in \eqref{200228e5} is
\begin{eqnarray*}
&&-2^{s+t+2}\sum_{|I|=s,|K|=t+1,|L|=t+1,|J|=t}'\sum_{i\neq j}\textbf{a}^{I,K} a_j^2\int_{\mathbb{C}^n}\varepsilon_{L}^{j,J}\varepsilon_{i,J}^{K}\frac{\partial f_{I,K}}{\partial \overline{z_j}}\cdot\overline{\frac{\partial f_{I,L}}{\partial \overline{z_i}}}\,\mathrm{d}\mathcal {N}^n\\
&=&-2^{s+t+2}\sum_{|I|=s,|K|=t+1,|L|=t+1,|J|=t}'\sum_{i\neq j}\textbf{a}^{I,K} a_j^2\int_{\mathbb{C}^n}\varepsilon_{i,J}^{K}\frac{\partial f_{I,K}}{\partial \overline{z_j}}\cdot\overline{\varepsilon_{L}^{j,J}\frac{\partial f_{I,L}}{\partial \overline{z_i}}}\,\mathrm{d}\mathcal {N}^n\\
&=&-2^{s+t+2}\sum_{|I|=s,|J|=t}'\sum_{ i\neq j}\textbf{a}^{I,J}\cdot a_j^2\cdot a_i^2\cdot\int_{\mathbb{C}^n}\frac{\partial f_{I,iJ}}{\partial \overline{z_j}}\cdot\overline{\frac{\partial f_{I,jJ}}{\partial \overline{z_i}}}\,\mathrm{d}\mathcal {N}^n.
\end{eqnarray*}

Combining the above computations, we arrive at
\begin{eqnarray*}
&&||T_n^*f||_{L^2_{(s,t)}(\mathbb{C}^n)}^2
+
||S_nf||_{L^2_{(s,t+2)}(\mathbb{C}^n)}^2\\
&=&2^{s+t+2}\sum_{|I|=s,|J|=t}' \sum_{1\leq i,j\leq n}\textbf{a}^{I,J}\cdot a_j^2\cdot a_i^2\cdot\int_{\Omega}\delta_j(f_{I,jJ})\cdot\overline{ \delta_i(f_{I,iJ})}\,\mathrm{d}\mathcal {N}^n\\
&&+2^{s+t+2}\sum_{|I|=s,|K|=t+1}'\sum_{ i\notin K}\textbf{a}^{I,K}\cdot a_i^2\cdot\int_{\mathbb{C}^n}\bigg|\frac{\partial f_{I,K}}{\partial \overline{z_i}}\bigg|^2\,\mathrm{d}\mathcal {N}^n\\
&&-2^{s+t+2}\sum_{|I|=s,|J|=t}'\sum_{ i\neq j}\textbf{a}^{I,J}\cdot a_j^2\cdot a_i^2\cdot\int_{\mathbb{C}^n}\frac{\partial f_{I,iJ}}{\partial \overline{z_j}}\cdot\overline{\frac{\partial f_{I,jJ}}{\partial \overline{z_i}}}\,\mathrm{d}\mathcal {N}^n\\
&=&2^{s+t+2}\sum_{|I|=s,|J|=t}'\sum_{1\leq i,j\leq n} \textbf{a}^{I,J}\cdot a_j^2\cdot a_i^2\cdot\int_{\Omega}\delta_j(f_{I,jJ})\cdot\overline{ \delta_i(f_{I,iJ})}\,\mathrm{d}\mathcal {N}^n\\
&&+2^{s+t+2}\sum_{|I|=s,|K|=t+1}'\sum_{ 1\leq i\leq n}\textbf{a}^{I,K}\cdot a_i^2\cdot\int_{\mathbb{C}^n}\bigg|\frac{\partial f_{I,K}}{\partial \overline{z_i}}\bigg|^2\,\mathrm{d}\mathcal {N}^n\\
&&-2^{s+t+2}\sum_{|I|=s,|J|=t}'\sum_{1\leq i,j\leq n}\textbf{a}^{I,J}\cdot a_j^2\cdot a_i^2\cdot\int_{\mathbb{C}^n}\frac{\partial f_{I,iJ}}{\partial \overline{z_j}}\cdot\overline{\frac{\partial f_{I,jJ}}{\partial \overline{z_i}}}\,\mathrm{d}\mathcal {N}^n\\
&=&2^{s+t+2}\sum_{|I|=s,|J|=t}' \sum_{1\leq i,j\leq n}\textbf{a}^{I,J} a_j^2 a_i^2\int_{\mathbb{C}^n}\Bigg(\delta_j(f_{I,jJ})\cdot\overline{ \delta_i(f_{I,iJ})}-\frac{\partial f_{I,jJ}}{\partial \overline{z_i}}\cdot\overline{\frac{\partial f_{I,iJ}}{\partial \overline{z_j}}}\Bigg)\,\mathrm{d}\mathcal {N}^n\\
&&+2^{s+t+2}\sum_{|I|=s,|K|=t+1}'\sum_{1\leq i\leq n}\textbf{a}^{I,K}\cdot a_i^2\cdot\int_{\mathbb{C}^n}\bigg|\frac{\partial f_{I,K}}{\partial \overline{z_i}}\bigg|^2\,\mathrm{d}\mathcal {N}^n,
\end{eqnarray*}
where the second equality follows from
\begin{eqnarray*}
2^{s+t+2}\sum_{|I|=s,|J|=t}'\sum_{1\leq i \leq n}\textbf{a}^{I,J}\cdot  a_i^4\cdot\int_{\mathbb{C}^n}\bigg|\frac{\partial f_{I,iJ}}{\partial \overline{z_i}}\bigg|^2\,\mathrm{d}\mathcal {N}^n\\
=2^{s+t+2}\sum_{|I|=s,|K|=t+1}'\sum_{ i\in K}\textbf{a}^{I,K}\cdot a_i^2\cdot\int_{\mathbb{C}^n}\bigg|\frac{\partial f_{I,K}}{\partial \overline{z_i}}\bigg|^2\,\mathrm{d}\mathcal {N}^n,
\end{eqnarray*}
and the third equality follows from
\begin{eqnarray*}
&&-2^{s+t+2}\cdot\sum_{|I|=s,|J|=t}'\sum_{1\leq i,j\leq n}\textbf{a}^{I,J}\cdot a_j^2\cdot a_i^2\cdot\int_{\mathbb{C}^n}\frac{\partial f_{I,iJ}}{\partial \overline{z_j}}\cdot\overline{\frac{\partial f_{I,jJ}}{\partial \overline{z_i}}}\,\mathrm{d}\mathcal {N}^n\\
&=&-2^{s+t+2}\cdot\sum_{|I|=s,|J|=t}' \sum_{1\leq i,j\leq n}\textbf{a}^{I,J}\cdot a_j^2\cdot a_i^2\cdot\int_{\mathbb{C}^n}\frac{\partial f_{I,jJ}}{\partial \overline{z_i}}\cdot\overline{\frac{\partial f_{I,iJ}}{\partial \overline{z_j}}}\,\mathrm{d}\mathcal {N}^n.
\end{eqnarray*}
This, together with the following two simple equalities:
\begin{eqnarray*}
\int_{\mathbb{C}^n}\frac{\partial \varphi}{\partial \overline{z_j}}\cdot\overline{\psi} \,\mathrm{d}\mathcal {N}^n
=-\int_{\mathbb{C}^n}\varphi\cdot\overline{\delta_j\psi}\,\mathrm{d}\mathcal {N}^n,\quad\,1\leq j\leq n,\,\forall\,\varphi,\psi\in\mathscr{C}_c^{\infty}(\mathbb{C}^n)
\end{eqnarray*}
and
\begin{eqnarray*}
\bigg(\delta_k\frac{\partial  }{\partial \overline{z_j}}-\frac{\partial  }{\partial \overline{z_j}}\delta_k\bigg)w=w\frac{1}{2a_k^2}\frac{\partial\overline{z_k}}{\partial \overline{z_j}},\quad\,1\leq k,j\leq n,\,\forall\,w\in \mathscr{C}_c^{\infty}(\mathbb{C}^n),
\end{eqnarray*}
yields that
\begin{eqnarray*}
&&||T_n^*f||_{L^2_{(s,t)}(\mathbb{C}^n)}^2
+
||S_nf||_{L^2_{(s,t+2)}(\mathbb{C}^n)}^2\\
&=&2^{s+t+2}\sum_{|I|=s,|J|=t}' \sum_{1\leq i,j\leq n}\textbf{a}^{I,J}\cdot a_j^2\cdot a_i^2\cdot\int_{\mathbb{C}^n}\frac{1}{2a_j^2}\cdot \frac{\partial\overline{z_j}}{\partial \overline{z_i}}\cdot f_{I,jJ}\cdot\overline{f_{I,iJ}}\,\mathrm{d}\mathcal {N}^n\\
&&+2^{s+t+2}\sum_{|I|=s,|K|=t+1}'\sum_{1\leq i\leq n}\textbf{a}^{I,K}\cdot a_i^2\cdot\int_{\mathbb{C}^n}\bigg|\frac{\partial f_{I,K}}{\partial \overline{z_i}}\bigg|^2\,\mathrm{d}\mathcal {N}^n\\
&=&(t+1)\cdot||f||_{L^2_{(s,t+1)}(\mathbb{C}^n)}^2+2^{s+t+1}\cdot\sum_{|I|=s,|K|=t+1}'\textbf{a}^{I,K}\cdot ||\overline{\partial}f_{I,K}||_{L^2_{(0,1)}(\mathbb{C}^n)}^2\\
&\geq & ||f||_{L^2_{(s,t+1)}(\mathbb{C}^n)}^2.
\end{eqnarray*}
This gives the desired estimate (\ref{200208e6}), hence the proof of Theorem \ref{200208t1} is complete.
\end{proof}

\begin{remark}
We believe that Theorems \ref{200207t1} and \ref{200208t1} should be equivalent but we cannot prove this equivalence at this moment.
\end{remark}

Now, as a consequence of Theorem \ref{200208t1}, we have the following key dimension-free {\it a priori} estimate:

\begin{corollary}\label{200208t2}
 It holds that
\begin{eqnarray}\label{200208e7}
||f||_{L^2_{(s,t+1)}(\mathbb{C}^n)}^2\leq ||T_n^*f||_{L^2_{(s,t)}(\mathbb{C}^n)}^2+||S_nf||_{L^2_{(s,t+2)}(\mathbb{C}^n)}^2,\,\quad\forall\,f\in D_{T_n^*}\cap D_{S_n}.
\end{eqnarray}
\end{corollary}

\begin{proof}
In finite dimensions, in view of \eqref{def of N^n}, the estimate (\ref{200208e6}) (in which all of the $L^2$ spaces $L^2_{(s,t)}(\mathbb{C}^n)$, $L^2_{(s,t+1)}(\mathbb{C}^n)$ and $L^2_{(s,t+2)}(\mathbb{C}^n)$ are defined in terms of Gaussian measure) can be equivalently re-written as a weighted estimate (in which the corresponding $L^2$ spaces are defined in terms of Legesgue measure). Hence, by (\ref{200208e6}), a standard smooth approximation argument implies that the inequality \eqref{200208e7} holds  (e.g., \cite[Chapter 4]{Kra} and \cite[Lemma 4.1.3]{Hor90}). This completes the proof of Corollary \ref{200208t2}.
\end{proof}

Now, combining Corollaries \ref{estimation lemma} and \ref{200208t2}, we immediately obtain the following result, which will play a fundamental role in the next section.

\begin{corollary}\label{200208t21}
 $R_{T_n}=N_{S_n}$ and for each $f\in N_{S_n}$ there exists $u\in D_{T_n}$ such that $T_nu=f$ and
 $$||u||_{_{L^2_{(s,t)}(\mathbb{C}^n)}}\leq ||f||_{_{L^2_{(s,t+1)}(\mathbb{C}^n)}}.
 $$
\end{corollary}

\begin{remark}\label{remark2-9-2}
By the conclusion 2) in Corollary \ref{estimation lemma}, in order to prove Corollary \ref{200208t21}, one uses only a weak version of \eqref{200208e7}, i.e.,
\begin{eqnarray}\label{weakestimation111lower bounded}
||f||_{L^2_{(s,t+1)}(\mathbb{C}^n)}\leq ||T_n^*f||_{L^2_{(s,t)}(\mathbb{C}^n)},\,\quad\forall\,f\in D_{T_n^*} \cap N_{S_n}.
\end{eqnarray}
\end{remark}

\section{Solving the $\overline{\partial}$ equation on $\ell^p$}\label{Sec5}

We are  now in a position to solve the $\overline{\partial}$ equation \eqref{d-bar equation} on $\ell^p$. In what follows, we fix two non-negative integers $s$ and $t$.

For each $n\in\mathbb{N}$, by Remark \ref{identify remark}, $L^2(\mathbb{C}^n,\mathcal {N}^n)$ can be identified as a closed subspace of $L^2(\ell^p,P)$.  We denote  by $\Pi_n$ the projection from $L^2(\ell^p,P)$ into $L^2(\mathbb{C}^n,\mathcal {N}^n)$.  For any $f\in N_S$ (the kernel of the operator $S$) with
\begin{eqnarray}\label{200214e1}
f= \sum_{|I|=s,|J|=t+1}'f_{I,J}\,\mathrm{d}z^I\wedge \mathrm{d}\overline{z}^J,
\end{eqnarray}
by the definition of $S$, for multi-indices with strictly increasing order $I=(i_1,\cdots,i_s)$ and $M=(m_1,\cdots,m_{t+2})$, we have
\begin{eqnarray}\label{kernel condition}
-\int_{\ell^p}\sum_{1\leq j<\infty}\sum_{|J|=t+1}'\varepsilon_{j, J}^{M}\cdot f_{I,J}\cdot\overline{\delta_j\varphi}\,\mathrm{d}P=0,\quad\forall\,\varphi\in\mathscr {C}_c^{\infty}.
\end{eqnarray}
Write
\begin{eqnarray*}
M_nf\triangleq \sum_{|I|=s,|J|=t+1,\, \max (I\cup J)\leq n}'\Pi_nf_{I,J}\,\mathrm{d}z^I\wedge \mathrm{d}\overline{z}^J.
\end{eqnarray*}
Then by Remark \ref{identify remark}, $M_nf$ can be identified as an element of $L^2_{(s,t+1)}(\mathbb{C}^n)$. Obviously, $M_nf\to f$ in $L^2_{(s,t+1)}(\ell^p)$ as $n\to\infty$.

We need the following simple result:

\begin{lemma}\label{200209l1}
For any $f\in N_S$, it holds that
$$M_nf\in N_{S_n}.
$$
\end{lemma}

\begin{proof}
For multi-indices $I=(i_1,\cdots,i_s)$ and $M=(m_1,\cdots,m_{t+2})$ with strictly increasing order such that $\max (I\cup M)\leq n$, one has $\varepsilon_{j, J}^{M}=0$ when $j>n$ or $\max J>n$. Suppose that $f$ takes the form of \eqref{200214e1}. By (\ref{kernel condition}) and the definition of $M_nf$, for each $\varphi\in\mathscr {C}_c^{\infty}(\mathbb{C}^n)$, we have
\begin{eqnarray*}
0&=&-\int_{\ell^p}\sum_{1\leq j<\infty}\sum_{|J|=t+1}'\varepsilon_{j, J}^{M}\cdot f_{I,J}\cdot\overline{\delta_j\varphi}\,\mathrm{d}P\\
&=&-\int_{\ell^p}\sum_{1\leq j\leq n}\sum_{|J|=t+1,\,\max J\leq n}'\varepsilon_{j, J}^{M}\cdot f_{I,J}\cdot\overline{\delta_j\varphi}\,\mathrm{d}P\\
&=&-\int_{\ell^p}\sum_{1\leq j\leq n}\sum_{|J|=t+1,\,\max J\leq n}'\varepsilon_{j, J}^{M}\cdot \Pi_nf_{I,J}\cdot\overline{\delta_j\varphi}\,\mathrm{d}P\\
&=&-\int_{\mathbb{C}^n}\sum_{1\leq j\leq n}\sum_{|J|=t+1,\,\max J\leq n}'\varepsilon_{j, J}^{M}\cdot \Pi_nf_{I,J}\cdot\overline{\delta_j\varphi}\,\mathrm{d}\mathcal {N}^n,
\end{eqnarray*}
where the third equality follows from the simple fact that $\delta_j\varphi\in L^2(\mathbb{C}^n,\mathcal{N}^n)$ and hence $\int_{\ell^p} f_{I,J}\cdot\overline{\delta_j\varphi}\,\mathrm{d}P=\int_{\ell^p}\Pi_nf_{I,J}\cdot\overline{\delta_j\varphi}\,\mathrm{d}P$. Hence, $M_nf\in N_{S_n}$.  This completes the proof of Lemma \ref{200209l1}.
\end{proof}

Now, by Lemma \ref{200209l1} and Corollary \ref{200208t21}  in previous section, there exists $u_n\in D_{T_n}$ such that $T_nu_n=M_nf$ and
\begin{equation}\label{200209e3}
||u_n||_{L^2_{(s,t)}(\mathbb{C}^n)}\leq ||M_nf||_{L^2_{(s,t+1)}(\mathbb{C}^n)}=||M_nf||_{L^2_{(s,t+1)}(\ell^p)}\leq ||f||_{L^2_{(s,t+1)}(\ell^p)}.
\end{equation}
Since $L^2_{(s,t)}(\mathbb{C}^n)$ can be naturally identified as a closed subspace of $L^2_{(s,t)}(\ell^p)$ and $||u_n||_{L^2_{(s,t)}(\mathbb{C}^n)}=||u_n||_{L^2_{(s,t)}(\ell^p)}$, combining the separability of $L^2_{(s,t)}(\ell^p)$ and  the Banach-Alaoglu Theorem (e.g., \cite[Theorems 3.15 and 3.17]{Rud91}), we conclude that, there exists a subsequence $\{u_{n_k}\}_{k=1}^{\infty}$ and $u\in L^2_{(s,t)}(\ell^p)$ such that
\begin{equation}\label{200209e1}
\lim\limits_{k\to\infty}(u_{n_k},g)_{L^2_{(s,t)}(\ell^p)}=(u,g)_{L^2_{(s,t)}(\ell^p)},\quad\forall\;g\in L^2_{(s,t)}(\ell^p).
\end{equation}

The following theorem is the main result in this paper.

\begin{theorem}\label{existence theorem}
For any $f\in N_S$, the form $u\in L^2_{(s,t)}(\ell^p)$ given above satisfies that
$$
u\in D_T,\qquad Tu=f
$$
and
\begin{equation}\label{200209e5}
||u||_{L^2_{(s,t)}(\ell^p)}\leq ||f||_{L^2_{(s,t+1)}(\ell^p)}.
 \end{equation}
\end{theorem}

\begin{proof}
Suppose $I=(i_1,\cdots,i_s)$, $J=(j_1,\cdots,j_t)$ and $K=(k_1,\cdots,k_{t+1})$ are multi-indices with strictly increasing order.  For each $n\in \mathbb{N}$, suppose that the solution $u_n\in D_{T_n}$ to $T_nu_n=M_nf$ given above takes the following form:
\begin{eqnarray*}
u_n= \sum_{|I|=s,|J|=t\, \max (I\cup J)\leq n}'u_{I,J}^n\,\mathrm{d}z^I\wedge \mathrm{d}\overline{z}^J,
\end{eqnarray*}
for some $u_{I,J}^n\in L^2(\ell^p)$. Also, write
\begin{eqnarray*}
u= \sum_{|I|=s,|J|=t}'u_{I,J}\,\mathrm{d}z^I\wedge \mathrm{d}\overline{z}^J,
\end{eqnarray*}
where $u_{I,J}\in L^2(\ell^p)$. By \eqref{200209e1},  it is easy to see that
\begin{equation}\label{200209e6}
\lim\limits_{k\to\infty}(u_{I,J}^{n_k},g)_{L^2(\ell^p)}=(u_{I,J},g)_{L^2(\ell^p)},\quad\forall\;g\in L^2(\ell^p).
\end{equation}

For $\varphi\in \mathscr{C}_c^{\infty}$, there exists $m\in\mathbb{N}$ such that $\varphi\in \mathscr{C}_c^{\infty}(\mathbb{C}^m)$. Set $n\triangleq \max(\max(I\cup K),m)$. We have
\begin{eqnarray}\label{key existence formula}
&&-\int_{\ell^p}\sum_{1\leq j<\infty}\sum_{|J|=t}'\varepsilon_{j, J}^{K}u_{I,J}\cdot\overline{\delta_j\varphi} \,\mathrm{d}P\\
&=&-\int_{\ell^p}\sum_{1\leq j\leq n}\sum_{|J|=t,\,\max J\leq n}'\varepsilon_{j, J}^{K}u_{I,J}\cdot\overline{\delta_j\varphi} \,\mathrm{d}P\nonumber\\
&=&-\sum_{1\leq j\leq n}\sum_{|J|=t,\,\max J\leq n}'\varepsilon_{j, J}^{K}\cdot\int_{\ell^p}u_{I,J} \cdot\overline{\delta_j\varphi} \,\mathrm{d}P.\nonumber
\end{eqnarray}
For each $n_k\geq n$ one can show that
\begin{eqnarray*}
&&-\sum_{1\leq j\leq n}\sum_{|J|=t,\,\max J\leq n}'\varepsilon_{j, J}^{K}\cdot\int_{\ell^p}u_{I,J}^{n_k}  \cdot\overline{\delta_j\varphi} \,\mathrm{d}P\\
&=&-\sum_{1\leq j\leq n_k}\sum_{|J|=t,\,\max J\leq n_k}'\varepsilon_{j, J}^{K}\cdot\int_{\mathbb{C}^{n_k}}u_{I,J}^{n_k}  \cdot\overline{\delta_j\varphi} \,\mathrm{d}\mathcal {N}^{n_k}\\
&=& \int_{\mathbb{C}^{n_k}}\Pi_{n_k}f_{I,K} \cdot\overline{\varphi} \,\mathrm{d}\mathcal {N}^{n_k}
=\int_{\ell^p}\Pi_{n_k}f_{I,K} \cdot\overline{\varphi} \,\mathrm{d}P,
\end{eqnarray*}
where the second equality follows from $T_{n_k}u_{n_k}=M_{n_k}f$ and an argument similar to the proof of \eqref{noncompact test}. Letting $k\to\infty$ in the above equality and noting \eqref{200209e6}, we obtain that
\begin{eqnarray*}
-\sum_{1\leq j\leq n}\sum_{|J|=t,\,\max J\leq n}'\varepsilon_{j, J}^{K}\cdot\int_{\ell^p}u_{I,J}  \cdot\overline{\delta_j\varphi} \,\mathrm{d}P=\int_{\ell^p} f_{I,K} \cdot\overline{\varphi} \,\mathrm{d}P.
\end{eqnarray*}
Combining this with (\ref{key existence formula}), we arrive at
\begin{eqnarray*}
-\int_{\ell^p}\sum_{1\leq j<\infty}\sum_{|J|=t}'\varepsilon_{j, J}^{K}u_{I,J}\cdot\overline{\delta_j\varphi} \,\mathrm{d}P=\int_{\ell^p} f_{I,K} \cdot\overline{\varphi} \,\mathrm{d}P.
\end{eqnarray*}
By the definition of $T$ at (\ref{weak def of st form}), the arbitrariness of $\varphi, I$ and $K$, we conclude that $Tu=f$.

Finally, the desired estimate \eqref{200209e5} follows from \eqref{200209e3} and \eqref{200209e1}. This completes the proof of Theorem \ref{existence theorem}.
\end{proof}

Similarly to the case of finite dimensions, we introduce the following notion:

\begin{definition}
A form $v\in L^2_{(s,t)}(\ell^p)$ is called analytic if $v\in D_T$ and $Tv=0$.
\end{definition}

Denote by $\mathscr{A}_{(s,t)}(\ell^p)$ the set of all analytic $L^2_{(s,t)}(\ell^p)$-forms. As a consequence of Theorem \ref{existence theorem}, we have the following result:

\begin{corollary}\label{uniqueness theorem}
For any $f\in N_S$, the solution $u\in L^2_{(s,t)}(\ell^p)$ given in Theorem \ref{existence theorem} is unique in the space $D_T\cap N_T^\bot$, and  the set of solutions to the $\overline{\partial}$ equation \eqref{d-bar equation} on $\ell^p$ is as follows:
$$
U\equiv\{u+v:v\in \mathscr{A}_{(s,t)}(\ell^p)\}.
$$
Moreover, $u\bot \mathscr{A}_{(s,t)}(\ell^p)$.
\end{corollary}

\begin{proof}
By means of Lemma \ref{lower bounded lemma} and Theorem \ref{existence theorem}, noting Remark \ref{estimation rem}, we obtain Corollary \ref{uniqueness theorem} immediately.
\end{proof}

Another consequence of Theorem \ref{existence theorem} is the following interesting $L^2$ estimate for general $(s,t+1)$-forms on $\ell^p$:
\begin{corollary}\label{non-direct-L2-estimate}
It holds that
\begin{eqnarray}\label{200209tt8}
||f||_{L^2_{(s,t+1)}(\ell^p)}\leq ||T^*f||_{L^2_{(s,t)}(\ell^p)},\quad\forall\,f\in N_{S}\cap D_{T^*}.
\end{eqnarray}
\end{corollary}

\begin{proof}
By means of the conclusion 2) in Corollary \ref{estimation lemma} and using the solvability result in Theorem \ref{existence theorem}, we obtain Corollary \ref{non-direct-L2-estimate} immediately.
\end{proof}

\begin{remark}
Unlike the case of finite dimensions, the estimate \eqref{200209tt8} in Corollary \ref{non-direct-L2-estimate} follows from the solvability result in Theorem \ref{existence theorem} (Recall that, the solvability of the finite dimensional $\overline{\partial}$ equation follows from $L^2$ estimates in the style of \eqref{weakestimation111lower bounded} in Remark \ref{remark2-9-2}). At this moment, it is unclear for us whether it is possible to prove the estimate \eqref{200209tt8} directly without using Theorem \ref{existence theorem}.
\end{remark}

\begin{example}\label{xmple1}
Let us revisit below Coeur\'{e} and Lempert's counterexample (in \cite{Lem99}) of $(0,1)$-form $f$ of class $C^{p-1}$ on $\ell^p$ for which the $\overline{\partial}$ equation $(\ref{d-bar equation})$ admits no (continuous) solution on any nonempty open sets of $\ell^p$, where $p\in \mathbb{N}$. In \cite[Section 9]{Lem99}, Lempert constructed the following $\overline{\partial}$-closed $(0,1)$-form on $\ell^p$:
\begin{eqnarray*}
f_0(\textbf{z},\textbf{z}^1)=\sum_{n=1}^{\infty}\psi(z_n)\overline{z_n^1},\,\quad\forall\,\textbf{z}=(z_n)=(x_n+\sqrt{-1}y_n),\,\textbf{z}^1=(z_n^1)\in\ell^p,
\end{eqnarray*}
where $\psi\in C^{p-1}(\mathbb{C})$ is any given function with compact support in $\mathbb{C}$ such that
$$
\left\{
\begin{array}{ll}
|\psi(z)|\leq |z|^{p-1},\quad \forall\; z\in\mathbb{C}, \\[3mm]
\psi(0)=0, \\[3mm]
\psi(z)=\frac{z^p}{\overline{z}\log |z|^2},\quad \hbox{for }0<|z|<1.
\end{array}
\right.
$$
A direct computations shows that
\begin{eqnarray*}
\sum_{n=1}^{\infty}2a_n^2\int_{\ell^p}|\psi(z_n)|^2\,\mathrm{d}P
&\leq& \sum_{n=1}^{\infty}2a_n^2\int_{\ell^p}|z_n|^{2(p-1)}\,\mathrm{d}P\\
&\leq& \sum_{n=1}^{\infty}2^pa_n^2\int_{\mathbb{R}^2}(|x_n|^{2(p-1)}+|y_n|^{2(p-1)})\,\mathrm{d}\mathcal {N}_{a_n}\\
&\leq&\frac{2^{2p}\Gamma(\frac{2p-1}{2})}{\sqrt{\pi}}\sum_{n=1}^{\infty}a_n^{2p}
<\frac{2^{2p}\Gamma(\frac{2p-1}{2})}{\sqrt{\pi}}\sum_{n=1}^{\infty}a_n<\infty,
\end{eqnarray*}
which implies that $f_0\in L^2_{(0,1)}(\ell^p)$. Recall that the definition of the $\overline{\partial}$ operator at (\ref{weak def of st form})-(\ref{d-dar-f-defi}) for $(0,1)$-form is the following: Suppose that $g\in L^2_{(0,1)}(\ell^p)$ has the following expansion:
\begin{eqnarray*}
g(\textbf{z},\textbf{z}^1)=\sum_{n=1}^{\infty}g_n(\textbf{z})\overline{z_n^1},\,\quad\forall\,\textbf{z}=(z_n),\,\textbf{z}^1=(z_n^1)\in\ell^p.
\end{eqnarray*}
If for each $1\leq i<j<\infty$, there exists $G_{ij}\in L^2(\ell^p,P)$ such that
\begin{eqnarray}
\int_{\ell^p}G_{ij}\overline{\varphi}\,\mathrm{d}P=-\int_{\ell^p}\big(g_{j}\cdot\overline{\delta_i\varphi} -g_{i}\cdot\overline{\delta_j\varphi} \big)\,\mathrm{d}P,\quad\forall\;\varphi\in \mathscr{C}_c^{\infty},
\end{eqnarray}
and $\sum\limits_{1\leq i<j<\infty} 4a_i^2a_j^2\int_{\ell^p}|G_{ij}|^2\,\mathrm{d}P<\infty$, then
\begin{eqnarray*}
\overline{\partial}g(\textbf{z},\textbf{z}^1,\textbf{z}^{2})=\sum_{ 1\leq i<j<\infty, J=(i,j)}G_{ij}(\textbf{z})(\mathrm{d}\overline{z}^{J})(\textbf{z}^1,\textbf{z}^{2}),\,\quad\forall\,\textbf{z},\,\textbf{z}^1,\,\textbf{z}^2\in\ell^p.
\end{eqnarray*}
Note that for each $\varphi\in \mathscr{C}_c^{\infty}$, positive integer $p>1$ and $1\leq i<j<\infty$, we have
\begin{eqnarray*}
\int_{\ell^p}\big(\psi(z_{j})\cdot\overline{\delta_i\varphi} -\psi(z_{i})\cdot\overline{\delta_j\varphi} \big)\,\mathrm{d}P
=-\int_{\ell^p}\bigg(\frac{\partial\psi(z_{j})}{\partial \overline{z_i}}\cdot\overline{ \varphi} -\frac{\partial\psi(z_{i})}{\partial \overline{z_j}}\cdot\overline{ \varphi} \bigg)\,\mathrm{d}P=0.
\end{eqnarray*}
When $p=1$, a simple computation shows that $\lim\limits_{|z|\to 0}|z|\cdot\big|\frac{\psi(z)}{\partial \overline{z}}\big|=0$ and $\psi$ is $C^{\infty}$ in $\{z\in \mathbb{C}:0<|z|<1\}$. Hence, there exists $\{\psi_k\}_{k=1}^{\infty}\subset C_c^{\infty}(\mathbb{C})$ such that
$$
\left\{
\begin{array}{ll}
0\leq \psi_k\leq 1,\\[3mm]
\psi_k(z)=0,\quad\forall\,|z|<\frac{1}{k}, \\[3mm]
\lim\limits_{k\to\infty}\psi_k(z)=1,\quad\forall\,z\in\mathbb{C}.
\end{array}
\right.
$$
Then $\{\psi_k\psi\}_{k=1}^{\infty}\subset C_c^{1}(\mathbb{C})$ and
\begin{eqnarray*}
&&\int_{\ell^1}\big(\psi_k(z_{j})\psi(z_{j})\cdot\overline{\delta_i\varphi} -\psi_k(z_{i})\psi(z_{i})\cdot\overline{\delta_j\varphi} \big)\,\mathrm{d}P\\
&=&-\int_{\ell^p}\bigg(\frac{\partial(\psi_k(z_{j})\psi(z_{j}))}{\partial \overline{z_i}}\cdot\overline{ \varphi} -\frac{\partial(\psi_k(z_{i})\psi(z_{i}))}{\partial \overline{z_j}}\cdot\overline{ \varphi} \bigg)\,\mathrm{d}P=0,
\end{eqnarray*}
letting $k\to\infty$ in the above, for $1\leq i<j<\infty$, we have
\begin{eqnarray*}
\int_{\ell^1}\big(\psi(z_{j})\cdot\overline{\delta_i\varphi} -\psi(z_{i})\cdot\overline{\delta_j\varphi} \big)\,\mathrm{d}P
=0.
\end{eqnarray*}
Hence, $f_0\in N_{\overline{\partial}}$. Now, by our Theorem \ref{existence theorem}, we conclude that there exists $u_0\in D_{\overline{\partial}}\subset L^2(\ell^p,P)$ such that $\overline{\partial}u_0=f_0$ for each positive integer $p$.
\end{example}

To end this section, denote by $H^2(\ell^2,P)$ the closure of the subspace $\text{span}\;\{ \textbf{z}^{\alpha} :\alpha \in \mathbb{N}^{(\mathbb{N})}\}$ in $L^2(\ell^2,P)$. Stimulated by the so-called monomial expansions of analytic functions as in \cite{DMP, Lem99, Rya}, $H^2(\ell^2,P)$ can be viewed as an infinite dimensional Hardy space (in the $L^2$ level).
Further, for non-negative integers $s$ and $t$, let $I=(1,2,\cdots,s)$, $J=(1,2,\cdots,t,t+1)$ and
\begin{eqnarray}\label{ex of (s,t+1)-form}
f(\textbf{z},\textbf{z}^1,\cdots,\textbf{z}^{s+t+1})= g(\textbf{z})(\mathrm{d}z^I\wedge \mathrm{d}\overline{z}^J)(\textbf{z}^1,\cdots,\textbf{z}^{s+t+1}),
\end{eqnarray}
where $\textbf{z},\,\textbf{z}^l=(z_j^l)\in\ell^p,\,\,1\leq l\leq s+t+1$ and $g\in L^2(\ell^p,P)$. If $g\in\mathscr{C}_c^{\infty}$, then $f\in\mathscr{D}_{(s,t+1)}$ and by (\ref{formula of d-bar})
\begin{eqnarray*}
\overline{\partial}f= (-1)^{s+t+1} \sum_{ t+2\leq j<\infty, J_j=(1,2,\cdots,t+1,j)} \frac{\partial g(\textbf{z})}{\partial \overline{z_j}}(\mathrm{d}z^I\wedge \mathrm{d}\overline{z}^{J_j})(\textbf{z}^1,\cdots,\textbf{z}^{s+t+2})\in \mathscr{D}_{(s,t+2)}.
\end{eqnarray*}
We have the following result:

\begin{proposition}\label{Hardy test example}
If $g\in H^2(\ell^2,P)$ and $f$ is given as in (\ref{ex of (s,t+1)-form}), then $f\in \mathscr{A}_{(s,t+1)}(\ell^p)$.
\end{proposition}
\begin{proof}
By $g\in H^2(\ell^2,P)$, there exists $\{g_n\}_{n=1}^{\infty}\subset \text{span}\;\{ \textbf{z}^{\alpha} :\alpha \in \mathbb{N}^{(\mathbb{N})}\}$ such that $\lim\limits_{n\to\infty}g_n=g$ in $L^2(\ell^p,P)$. Hence, for each $\varphi\in \mathscr{C}_c^{\infty}$ and $t+2\leq j<\infty$, let $ J_j=(1,2,\cdots,t+1,j)$ and we have
\begin{eqnarray*}
-\int_{\ell^p}\sum_{1\leq i<\infty} \varepsilon_{i, J}^{J_j}g\cdot\overline{\delta_i\varphi} \,\mathrm{d}P
&=&(-1)^{t+2} \int_{\ell^p}g\cdot\overline{\delta_j\varphi} \,\mathrm{d}P\\
&=&(-1)^{t+2}\lim\limits_{n\to\infty}\int_{\ell^p}g_n\cdot\overline{\delta_j\varphi} \,\mathrm{d}P\\
&=&(-1)^{t+3}\lim\limits_{n\to\infty}\int_{\ell^p}\frac{\partial g_n}{\partial \overline{z_j}}\cdot\overline{ \varphi} \,\mathrm{d}P=0.
\end{eqnarray*}
By (\ref{weak def of st form}), we conclude that $f\in\mathscr{A}_{(s,t+1)}(\ell^p)$. This completes the proof of Proposition \ref{Hardy test example}.
\end{proof}

\section*{Acknowledgement}
This work is partially supported by NSF of China under grants 11931011 and 11821001.


\end{document}